\documentclass[11pt]{article}
\usepackage{amscd, amsmath, amssymb, amsthm, mathtools}
\usepackage[all,cmtip]{xy}
\usepackage[pagebackref]{hyperref}
\usepackage{verbatim}

\title{Log canonical pairs with conjecturally minimal volume}
\author{Louis Esser and Burt Totaro}

\date{  }

\def\Q{\text{\bf Q}}
\def\R{\text{\bf R}}
\def\C{\text{\bf C}}
\def\P{\text{\bf P}}

\DeclareMathOperator{\vol}{vol}
\DeclareMathOperator{\Aut}{Aut}
\DeclareMathOperator{\mld}{mld}
\def\T{\mathsf{T}}
\def\sing{\text{sing}}

\begin{document}
\maketitle
\newtheorem{theorem}{Theorem}[section]
\newtheorem{proposition}[theorem]{Proposition}
\newtheorem{corollary}[theorem]{Corollary}
\newtheorem{lemma}[theorem]{Lemma}

\theoremstyle{definition}
\newtheorem{definition}[theorem]{Definition}
\newtheorem{example}[theorem]{Example}
\newtheorem{conjecture}[theorem]{Conjecture}

\theoremstyle{remark}
\newtheorem{remark}[theorem]{Remark}

\numberwithin{equation}{section}

Consider the problem of finding a complex projective log canonical pair
$(X,B)$ with $B$ a nonzero reduced divisor and $K_X+B$ ample
such that the volume of $K_X+B$ is as small as possible.
This problem arises naturally in attempts to classify
stable varieties of general type \cite[Remark 7.10]{LS}.
We know that there is some positive lower
bound for the volume in each dimension,
by Hacon-M\textsuperscript{c}Kernan-Xu \cite[Theorem 1.6]{HMXbounded}.

V.~Alexeev and W.~Liu constructed a log canonical pair $(X,B)$ of dimension 2
with $B$ a nonzero reduced divisor and $K_X+B$ ample
such that $K_X+B$ has volume $1/462$ \cite[Theorem 1.4]{AL}.
J.~Liu and V.~Shokurov showed that this example is not at all arbitrary:
it has the smallest possible volume in dimension 2,
under the given conditions \cite[Theorem 1.4]{LS}. (See also Koll\'ar's
example in the broader class of log canonical pairs with standard
coefficients, Remark \ref{kollarexample}.)

In this paper, we give a simpler description of Alexeev-Liu's example:
it is a non-quasi-smooth hypersurface in a weighted projective space,
$X_{42}\subset \P^3(21,14,6,11)$, with $B$ the curve $\{ x_3=0\}
\cap X$. (This fits into a remarkable number of classification
problems in algebraic geometry for which the extreme case is known
or conjectured to be a weighted hypersurface \cite{ETWgeneral,ETWCalabi}.)
We generalize that construction to produce a log canonical
pair $(X,B)$ of any dimension
with $B$ a nonzero reduced divisor such that $K_X+B$ is ample
and has extremely small volume.
We conjecture that our
example has the smallest possible volume of $K_X+B$ in every dimension.
The volume is doubly exponentially small in terms of the dimension.

A similar story was worked out earlier in the case $B=0$,
where smaller volumes can occur.
Namely, Alexeev and Liu constructed a projective klt surface
with ample canonical class and volume $1/48983$ \cite[Theorem 1.4]{AL}.
Totaro found that their surface is a non-quasi-smooth hypersurface
in a weighted projective space,
$X_{438}\subset \P^3(219,146,61,11)$. Generalizing that
construction, he produced a klt variety of each dimension
with ample canonical class and conjecturally minimal volume
\cite[Theorem 2.1]{Totaroklt}.

We also develop examples for some related extremal problems.
Esser constructed
a klt Calabi--Yau variety which conjecturally has the smallest minimal
log discrepancy
in each dimension \cite[Conjecture 4.4]{Essermld}. (In particular,
this variety has mld $1/13$ in dimension 2, $1/311$
in dimension 3, and $1/677785$ in dimension 4. In dimension 2,
we know that $1/13$ is the smallest possible mld
\cite[Proposition 6.1]{ETWCalabi}.)
However,
the properties of Esser's example were not 
worked out in all dimensions.
We now prove the desired properties of Esser's example
(in particular, determining its mld $1/m$),
as Theorem \ref{kltvariety}.
Using this example, it follows
that the ``first gap of global lc thresholds'' (in Liu-Shokurov's
terminology) is at most
the same number $1/m$, meaning that there is a klt Calabi--Yau pair
$(X,(1-\frac{1}{m})S)$ with $S$ an irreducible divisor.
We present such a ``pair'' variant of Esser's example
explicitly in Theorem \ref{pairmld}.

Likewise, Wang and the authors constructed a klt Calabi--Yau variety
which conjecturally has the largest index in each dimension
\cite[Conjecture 7.10]{ETWCalabi}. This variety has 
index $19$ in dimension
$2$, $493$ in dimension $3$, and $1201495$ in dimension $4$.
(In dimension 2,
we know that $19$ is the largest possible index
\cite[Proposition 6.1]{ETWCalabi}.)
The numerics in this example
are extremely similar to Esser's, and so we can now give clearer proofs of 
many of its properties (section \ref{largeindexintrosection}).
Nevertheless, the precise formula for the index 
depends on a conjecture that two explicit numbers
are relatively prime (Conjecture \ref{conj-index}), which holds
in dimensions at most 30.

In section \ref{asymptotics_section},
we compute asymptotics for the small mld
and large index examples that 
provide additional evidence for their optimality.

Totaro was supported by NSF grant DMS-2054553,
Simons Foundation grant SFI-MPS-SFM-00005512,
and the Charles Simonyi Endowment
at the Institute for Advanced Study.

\section{Notation}
\label{notation}

Our examples use {\it Sylvester's sequence}, defined
by $s_0=2$ and $s_{j+1}=s_j(s_j-1)+1$. The sequence
begins $2,3,7,43,1807,\ldots$. We have $s_{j+1}
=s_0\cdots s_j+1$, and hence the numbers in Sylvester's sequence
are pairwise coprime. The key property of this sequence is that
$$\frac{1}{s_0}+\cdots+\frac{1}{s_{j}}=1-\frac{1}{s_{j+1}-1}.$$
The sequence $s_j$ grows doubly exponentially, with
$s_j>2^{2^{j-1}}$ for all $j\geq 0$.

For positive integers $a_0,\ldots,a_n$, the weighted
projective space $Y=\P^n(a_0,\ldots,a_n)$ means
the quotient variety $(A^{n+1}-0)/G_m$ over $\C$, where
the multiplicative group $G_m$ acts by
$t(x_0,\ldots,x_n)=(t^{a_0}x_0,\ldots,t^{a_n}x_n)$
\cite[section 6]{Iano-Fletcher}. We say that $Y$ is
{\it well-formed }if $\gcd(a_0,\ldots,\widehat{a_j},\ldots,
a_n)=1$ for each $j$, which means 
that the $G_m$-action is free in codimension 1.
For a well-formed weighted projective space $Y$
and an integer $m$,
$O_Y(m)$ is the reflexive sheaf associated
to a Weil divisor. The divisor class
$O_Y(m)$ is Cartier if and only if $m$ is a multiple
of every weight $a_j$. Well-formedness of $Y$ ensures that the canonical
divisor is given by $K_Y=O_Y(-\sum a_j)$.

Let $Y$ be a well-formed weighted projective space.
A closed subvariety $X$ of $Y$ is called
{\it quasi-smooth }if its affine cone in $A^{n+1}$
is smooth outside the origin. In particular,
a quasi-smooth subvariety has only cyclic quotient
singularities and hence is Kawamata log terminal (klt).
(A reference for the singularities of the minimal model program
such as klt, plt, lc is \cite[Definition 2.8]{Kollarsing}.)
Also, $X$ is {\it well-formed }if
$Y$ is well-formed and the codimension of $X\cap Y^{\sing}$
in $X$ is at least 2. (For a well-formed weighted projective
space $Y$, the singular locus of $Y$
corresponds to the locus where the $G_m$-action is not free.)
Iano-Fletcher proved the following sufficient criterion for well-formedness
\cite[Theorem 6.17]{Iano-Fletcher}.

\begin{proposition}
\label{iano}
As long as the degree $d$
is not equal to any of the weights, every quasi-smooth hypersurface
of dimension at least $3$ in a well-formed weighted projective space
is well-formed.
\end{proposition}

For a well-formed normal hypersurface $X$ of degree $d$
in a weighted projective space $Y$, the canonical divisor
is given by $K_X=O_X(d-\sum a_j)$
\cite[section 1]{Totaroklt}. Here $X$ need not
be quasi-smooth.

A Weil divisor or more generally
a $\Q$-divisor is said to be {\it ample }if some positive multiple
is an ample Cartier divisor.
The {\it volume }of a $\Q$-divisor $D$
on a normal projective variety $X$ is
$$\vol(D) \coloneqq \lim_{m\to\infty} h^0(X,\lfloor mD\rfloor)/(m^n/n!),$$
where $n=\dim(X)$. (The volume in this generality is discussed
in \cite{FKL}.) The volume is equal to the intersection
number $D^n$ if $D$ is ample.
The volume of an ample Cartier divisor
is an integer, but that fails in general for an ample Weil divisor.
For example, the volume of $O_Y(1)$ on a well-formed weighted
projective space $Y=\P^n(a_0,\ldots,a_n)$ is $1/(a_0\cdots a_n)$.

A {\it pair }$(X,D)$ in this paper means
a normal variety $X$ with a $\Q$-divisor $D$ such that $K_X+D$
is $\Q$-Cartier. A pair
is {\it Calabi--Yau} if $D$ is effective and
$K_X+D$ is $\Q$-linearly equivalent to zero.
In that case, the {\it index }of
$(X,D)$ is the smallest positive integer $m$ such that $m(K_X+D)\sim 0$.
A pair has {\it standard coefficients }if each
coefficient of $D$ is of the form $1-\frac{1}{b}$
with $b$ a positive integer or $\infty$.
For a klt Calabi--Yau pair $(X,D)$ with standard coefficients
and index $m$, the (global) {\it index-1 cover }of $(X,D)$
is a projective variety $Y$ with canonical Gorenstein singularities
such that the canonical class $K_Y$ is linearly equivalent
to zero \cite[Example 2.47, Corollary 2.51]{Kollarsing}.
Here $(X,D)$ is the quotient of $Y$ by an action
of the cyclic group $\mu_m$ such that
$\mu_m$ acts faithfully on $H^0(Y,K_Y)\cong \C$.
(Explicitly, $D$ has coefficient $1-\frac{1}{b}$
on the image of an irreducible divisor on which the subgroup
of $\mu_m$ that acts as the identity has order $b$.)

For a pair $(X,D)$ and a proper birational morphism $\pi\colon Y\to X$
with $Y$ normal, there is a uniquely defined $\Q$-divisor $D_Y$ on $Y$
such that $K_Y+D_Y=\pi^*(K_X+D)$. The {\it log discrepancy }of $(X,D)$
with respect to an irreducible divisor $S$ on $Y$ is 1 minus the coefficient
of $S$ in $D_Y$. The {\it minimal log discrepancy }({\it mld}) of $(X,D)$
is the infimum of all log discrepancies of $(X,D)$
with respect to all irreducible divisors on all birational models of $X$.
Thus $(X,D)$ is klt if and only if its mld is positive,
and a pair with smaller mld can be considered more singular.

Berglund-H\"ubsch-Krawitz mirror symmetry considers
weighted-homogeneous polynomials of the following three basic types,
{\it Fermat}, {\it loop}, and {\it chain}, as well as combinations
of them in disjoint sets of variables \cite[section 2.2]{ABS}:
\begin{align*}
W_{\text{Fermat}}&=x^b,\\
W_{\text{loop}}&=x_1^{b_1}x_2+x_2^{b_2}x_3+\cdots +x_{n-1}^{b_{n-1}}x_n
+x_n^{b_n}x_1,\\
W_{\text{chain}}&=x_1^{b_1}x_2+x_2^{b_2}x_3+\cdots +x_{n-1}^{b_{n-1}}x_n
+x_n^{b_n}.
\end{align*}
The corresponding weighted projective hypersurfaces
are quasi-smooth
for any positive integers $b_i$, not all 1.

\section{A pair with nonzero boundary and small volume}

\begin{theorem}
\label{ample}
For each integer $n$ at least $2$, let 
$$a_{n+1} \coloneqq \begin{cases} \frac{1}{4}(s_n^2-s_n+2) &\text{if $n$ is even}\\
\frac{1}{4}(s_n^2-3s_n+4) &\text{if $n$ is odd.}
\end{cases}$$
Let $d \coloneqq s_{n+1}-1$, and $a_i \coloneqq d/s_i$ for $0\leq i\leq n$.
Then there is a complex hypersurface $X$ of degree $d$ in
$\P^{n+1}(a_0,\ldots,a_{n+1})$ that is well-formed and klt.
Let $B \coloneqq \{x_{n+1}=0\}\cap X$. Then the pair $(X,B)$ is purely log terminal
(hence log canonical),
$B$ is a nonzero reduced divisor,
$K_X+B$ is ample, and the volume of $K_X+B$ is
$$\frac{1}{(s_{n+1}-1)^{n-1}a_{n+1}},$$
which is asymptotic to $4/s_n^{2n}$.
In particular, this is less
than $1/2^{2^n}$.
\end{theorem}

The numerology here is similar but not identical to that
of the klt variety with ample canonical class and conjecturally
minimal volume. In particular, the latter example involves
the same weight $a_{n+1}$ \cite[Theorem 2.1]{Totaroklt}. For comparison,
the volume of $K_X$ in that example is asymptotic
to $2^{2n+2}/s_n^{4n}$, which is much smaller than the volume
of $K_X+B$ above. (Requiring $B$ to be a nonzero reduced
divisor forces the volume to be bigger, it seems.)

Explicitly, define the variety $X$ in Theorem \ref{ample}
by the equation, for $n\geq 2$ even:
$$0=x_0^2+x_1^3+\cdots
+x_n^{s_n}+x_1\cdots x_nx_{n+1}^2.$$
For $n\geq 3$ odd, define $X$ by
$$0=x_0^2+x_1^3+\cdots
+x_{n}^{s_n}+x_1\cdots x_{n-1}x_n^2x_{n+1}^2.$$
Also, define $B \coloneqq \{x_{n+1}=0\} \cap X$.
Since the number of monomials is equal
to the number of variables, any linear combination
of these monomials with all coefficients nonzero
defines a variety isomorphic to $X$, by scaling
the variables. One can check
that the monomials shown are all the monomials of degree $d$,
and hence that an open subset
of all hypersurfaces of degree $d$ are isomorphic to the one indicated;
but we will not need those facts.

Note that $X$ is not quasi-smooth.

\begin{conjecture}
For each integer $n$ at least 2,
the pair in Theorem \ref{ample}
has the smallest volume among all projective lc pairs $(X,B)$
of dimension $n$ with $B$ a nonzero reduced divisor
and $K_X+B$ ample.
\end{conjecture}

We know that there is some
positive lower bound for the volume in each dimension,
by Hacon-M\textsuperscript{c}Kernan-Xu \cite[Theorem 1.6]{HMXbounded}.

In dimension 2, our example is 
$X_{42}\subset \P^3(21,14,6,11)$ and $B=\{x_3=0\}\cap X$,
with $\vol(K_X+B)=1/462
\doteq 2.2\times 10^{-3}$. As discussed in the introduction,
this is the smallest possible volume for a projective lc pair $(X,B)$
of dimension 2 with $B$ a nonzero reduced divisor and $K_X+B$ ample,
by J.~Liu and V.~Shokurov \cite[Theorem 1.4]{LS}.
This example was found by V.~Alexeev and W.~Liu,
without the description as a hypersurface \cite[Theorem 1.4]{AL}.

In dimension 3,
our example is 
$$X_{1806}\subset \P^4(903,
602,258,42,431).$$
For $B=\{x_4=0\}\cap X$, we have $\vol(K_X+B)\doteq 7.1\times 10^{-10}$.
In dimension 4, our example is
$$X_{3263442}\subset \P^5(1631721,1087814,466206,75894,1806,
815861).$$
For $B=\{x_5=0\}\cap X$, we have $\vol(K_X+B)
\doteq 3.5\times 10^{-26}$.

\begin{remark}
\label{kollarexample}
In the broader class
of log canonical pairs $(X,B)$ such that $B$ has standard
coefficients with $\lfloor B\rfloor\neq 0$ and $K_X+B$ ample,
Koll\'ar conjectured the example of smallest volume
\cite[Example 5.3.1]{Kollarlog}: for general hyperplanes
$H_0,\ldots,H_{n+1}$ in $\P^n$, let
$$(X,B)=(\P^n,\frac{1}{2}H_0+\frac{2}{3}H_1+\cdots
+\frac{s_n-1}{s_n}H_n+H_{n+1}).$$
Here $K_X+B$ has volume $1/(s_{n+1}-1)^n$, which is asymptotic
to $1/s_n^{2n}$. In dimension 2, Koll\'ar showed that
this example is indeed optimal with these properties,
with volume $1/1764$ \cite[Remark 6.2.1]{Kollarlog}.

In high dimensions, Koll\'ar's pair has about $1/4$ of the volume
in Theorem \ref{ample}, which is extremely close for such small
numbers. That is some evidence for the optimality of Theorem
\ref{ample}, in the narrower setting of reduced divisors.
\end{remark}

\begin{proof}
(Theorem \ref{ample})
The weight $a_{n+1}$ is odd,
as we showed in \cite[proof of Theorem 2.1]{Totaroklt} (since our example
with $B=0$ included the same weight $a_{n+1}$ as here).
For the reader's convenience, here is the argument:
$s_n$ is $7\pmod{8}$ if $n\geq 2$ is even
and $3\pmod{8}$ if $n\geq 3$ is odd. This is immediate
by induction from the recurrence $s_{n+1}=s_n(s_n-1)+1$.
It follows that $s_n^2-s_n+2$ is $4\pmod{8}$
if $n\geq 2$ is even, and that $s_n^2-3s_n+4$
is $4\pmod{8}$ if $n\geq 3$ is odd. So $a_{n+1}$ is odd
in both cases.

Next, we show that the weighted projective
space $Y=\P^{n+1}(a_0,\ldots,a_{n+1})$ is well-formed.
That is, we have to show that $\gcd(a_0,\ldots,\widehat{a_j},
\ldots,a_{n+1})=1$ for each $j$.
We have $s_{n+1}-1=s_0\cdots s_n$
with $s_0,\ldots,s_n$ pairwise coprime, which implies that
$\gcd(a_0,\ldots,a_n)=1$. For the rest, it suffices to show
that $\gcd(s_{n+1}-1,a_{n+1})=1$. Here $s_{n+1}-1=s_n(s_n-1)$,
so it suffices to show that $\gcd(s_n-1,a_{n+1})=1$ and
$\gcd(s_n,a_{n+1})=1$. The first statement was shown
in \cite[proof of Theorem 2.1]{Totaroklt}. For the second,
if a prime number $p$ divides $s_n$ and $a_{n+1}$, then $p>2$
since $a_{n+1}$ is odd. If $n$ is even, it follows
that $a_{n+1}=\frac{1}{4}(s_n^2-s_n+2)\equiv \frac{1}{2}\pmod {p}$,
not 0, which is a contradiction. If $n$ is odd, we have
$a_{n+1}=\frac{1}{4}(s_n^2-3s_n+4)\equiv 1\pmod {p}$, not 0,
which is a contradiction. That completes the proof
that $Y$ is well-formed.

From the equation for $X$, the only coordinate linear space of $Y$
contained in $X$ is the point $p \coloneqq [0,\ldots,0,1]$. Since
that has codimension at least 2 in $X$, $X$ is well-formed.
Also, $X$ is quasi-smooth outside $p$, so it has only quotient
singularities outside $p$, and so $X$ is klt outside $p$.
At the point $p$, in coordinates $x_{n+1}=1$, $X$ is defined
by the equation
$$0=x_0^2+x_1^3+\cdots+x_n^{s_n}+x_1\cdots x_n$$
for $n$ even, resp.\
$$0=x_0^2+x_1^3+\cdots+x_n^{s_n}+x_1\cdots x_{n-1}x_n^2$$
for $n$ odd (in $A^{n+1}/\mu_{a_{n+1}}$). The same hypersurfaces
in $A^{n+1}$ appeared in our klt variety of conjecturally minimal volume,
and we showed that these hypersurfaces in $A^{n+1}$
have canonical singularities (hence are klt)
\cite[proof of Theorem 2.1]{Totaroklt}. The klt, plt, and lc
properties are preserved upon dividing
by a finite group action that is free in codimension 1
\cite[Corollary 2.43]{Kollarsing}. Therefore,
$X$ is klt.

The divisor $B=\{x_{n+1}=0\}\cap X$ is quasi-smooth and misses
the point $p$. So the pair $(X,B)$ is \'etale-locally (near each point
of $B$) the quotient of $(A^n,A^{n-1})$ by a finite group action,
free in codimension 1.
Since $X$ is klt outside $B$, it follows
that the pair $(X,B)$ is plt (hence lc).

Let $d=s_{n+1}-1$ be the degree of the hypersurface $X$.
Since $X$ is well-formed and normal, we have
\begin{align*}
K_X&=O_X(d-\sum a_j)\\
&=O_X\bigg( (s_{n+1}-1)\bigg(1-\frac{1}{s_0}-\cdots-\frac{1}{s_n}\bigg)
-a_{n+1}\bigg)\\
&=O_X(1-a_{n+1}).
\end{align*}
The divisor $B$ on $X$ is linearly equivalent to $O_X(a_{n+1})$,
and so $K_X+B\sim O_X(1)$. It follows that
\begin{align*}
\vol(K_X+B)&=\vol(O_X(1))\\
&=\frac{d}{a_0\cdots a_{n+1}}\\
&=\frac{(s_{n+1}-1)s_0\cdots s_n}{(s_{n+1}-1)^{n+1}a_{n+1}}\\
&=\frac{1}{(s_{n+1}-1)^{n-1}a_{n+1}}.
\end{align*}
Here $s_{n+1}$ is asymptotic to $s_n^2$ (with error term on the order
of $s_n$) as $n$ goes to infinity,
and $a_{n+1}\sim s_n^2/4$; so the volume of $K_X+B$ is asymptotic
to $4/s_n^{2n}$.
\end{proof}

\section{Esser's klt Calabi--Yau variety with small mld}
\label{esserintro}

Esser constructed
a klt Calabi--Yau variety which conjecturally has the smallest mld
(roughly $1/2^{2^n}$) in each dimension $n$
\cite[Conjecture 4.4]{Essermld}. (For example,
this variety has mld $1/13$ in dimension 2, $1/311$
in dimension 3, and mld $1/677785$ in dimension 4.)
We know that there is some positive lower bound for this problem
in each dimension,
by Hacon-M\textsuperscript{c}Kernan-Xu \cite[Theorem 1.5]{HMXACC}.

In version 1 of \cite{Essermld} on the arXiv,
the example was worked out completely only in dimensions
at most 18. In this paper, we prove the desired properties
of Esser's example in all dimensions, in particular
confirming Esser's conjectured value for its mld
(Theorem \ref{kltvariety}). By Lemma \ref{constant},
this mld is within a constant factor of the conjecturally smallest mld
in the broader setting of klt Calabi--Yau pairs with standard
coefficients.

We were led to this analysis by constructing a related example
among pairs, although in this paper we will prove
the properties of Esser's example first.
The example among pairs (Theorem \ref{pairmld}) is a klt Calabi--Yau
pair $(X,(1-\frac{1}{m})S)$ of each dimension $n\geq 2$ with $S$
an irreducible divisor. The number $1/m$ is the same as the mld
of Esser's example, and in fact $(X,(1-\frac{1}{m})S)$
is crepant-birational to Esser's Calabi--Yau variety
$V/\mu_m$ (Lemma \ref{crepant}).
(That is, $S$ is the divisor that shows that $V/\mu_m$
has mld $1/m$.)

In each dimension $n\geq 2$,
Esser's example is the quotient
of a hypersurface $V$ in a weighted projective
space $\P^{n+1}(a_0,\ldots,a_{n+1})$ by an action of a cyclic
group \cite[section 4]{Essermld}.
(The order of the cyclic group should be the number $m$
defined below, but that will be proved later,
in Theorem \ref{kltvariety}. We define the $\mu_m$-action
explicitly in the proof of Lemma \ref{crepant}.)
In odd dimension $n=2r+1$, the equation of $V$ has the form
$$0=x_0^{b_0}x_{2r+2}+x_1^{b_1}x_{2r+1}+\cdots+x_r^{b_r}x_{r+2}
+x_{r+1}^{b_{r+1}}x_r
+\cdots+x_{2r+1}^{b_{2r+1}}x_0
+x_{2r+2}^{v_{2r+1}}x_{r+1},$$
for exponents $b_j$ and $v_{2r+1}$ defined below.
In even dimension $n=2r$,
the equation of $V$ has the form
$$0=x_0^{b_0}+x_1^{b_1}x_{2r+1}+\cdots+x_r^{b_r}x_{r+2}
+x_{r+1}^{b_{r+1}}x_r
+\cdots+x_{2r}^{b_{2r}}x_1
+x_{2r+1}^{v_{2r}}x_{r+1}.$$
To define the exponents, we'll use the 
following notation. For short, given 
integers $b_{i_1},\ldots,b_{i_k}$,
write $B_{i_1\cdots i_k}$ for the alternating sum
$$B_{i_1\cdots i_k} \coloneqq b_{i_1}\cdots b_{i_k}-b_{i_1}\cdots b_{i_{k-1}}
+\cdots+(-1)^{k-1}b_{i_1}+(-1)^k.$$
We note for future reference the following
symmetry property for alternating sums of $B$'s:

\begin{lemma}
    \label{Bsymmetry}
    For any integers $b_{i_1},\ldots,b_{i_k}$,
    $$B_{i_1\cdots i_k} - B_{i_2 \cdots i_k} + \cdots + 
    (-1)^{k-1}B_{i_k} = B_{i_k \cdots i_1} - 
    B_{i_{k-1} \cdots i_1} + \cdots + (-1)^{k-1} B_{i_1}.$$
\end{lemma}

Esser defines the exponents $b_0,\ldots,b_{n}$
as follows, for $n=2r+1$ or $n=2r$,
with $r\geq 1$. For $0\leq i\leq r$, let $b_i \coloneqq s_i$, the Sylvester number.
Then define all but one of the remaining exponents inductively by
\begin{align*}
b_{r+i}& \coloneqq 1+(b_{r+1-i}-1)^2B_{r+1,r,r+2,r-1,\ldots,r-1+i,r+2-i}\\
&=1+b_0\cdots b_{r-i}(b_{r+1-i}-1)B_{r+1,r,r+2,r-1,\ldots,r-1+i,r+2-i}
\end{align*}
for $1\leq i\leq r+1$ when $n=2r+1$, 
and for $1\leq i\leq r$ when $n=2r$.
(A symbol $B$ with empty subscript
is understood to be 1, and so $b_{r+1}=1+(b_r-1)^2$.)
Finally, the last exponent is given by
$$\begin{cases} v_{2r+1} \coloneqq B_{r+1,r,r+2,r-1,\ldots,2r+1,0}-B_{r,r+2,r-1,\ldots,2r+1,0}+\cdots-B_0 &\text{if }n=2r+1,\\
v_{2r} \coloneqq 2(B_{r+1,r,r+2,r-1,\ldots,2r,1}-B_{r,r+2,r-1,\ldots,2r,1}+\cdots-B_1)+1 &\text{if }n=2r.
\end{cases}$$

The weights $a_0,\ldots,a_{n+1}$ and degree $D$ of $V$
are determined uniquely by the equation of $V$, given the requirement that
$\gcd(a_0,\ldots,a_{n+1})=1$. Esser shows
that $V$ is a well-formed quasi-smooth Calabi--Yau hypersurface;
we write out the details in section \ref{kltvarietysection}.
In version 1 of \cite[Section 4]{Essermld} on the arXiv,
Esser {\it conjectured }that the last
weight $a_{2r+2}$ is equal to 1, and he verified this in dimensions
at most 18. Using that assumption,
he shows that there is an action of the cyclic group $\mu_m$ on $V$,
where $m = m_n$ is given by
$$m \coloneqq \begin{cases} m_{2r+1} = B_{0,2r+1,1,2r,\ldots,r,r+1} &\text{if }n=2r+1,\\
m_{2r} = B_{1,2r,2,2r-1,\ldots,r,r+1} &\text{if }n=2r.
\end{cases}$$
The number $m=m_n$ is doubly exponential in the dimension $n$; 
in particular, $m_n>2^{2^n}$ for $n>2$.
Esser's conjecture would also imply that the degree $D$ of $V$
is given by $D=u_{2r+1}$ if $n=2r+1$ and $D=2u_{2r}$ if $n=2r$,
where
$$u \coloneqq \begin{cases}
u_{2r+1} = B_{r+1,r,r+2,r-1,\ldots,2r+1,0}&\text{if }n=2r+1,\\
u_{2r} = B_{r+1,r,r+2,r-1,\ldots,2r,1} &\text{if }n=2r.
\end{cases}$$
(A connection between dimensions $2r$
and $2r+1$ is that the exponent $b_{2r+1}$ for $n=2r+1$
satisfies $b_{2r+1}=u_{2r}+1$.)

Esser's conjecture would imply that $V/\mu_m$ is a klt Calabi--Yau variety
with mld $1/m$; so this was initially proved
in dimensions $n\leq 18$. We prove the conjecture in all dimensions
in Theorem \ref{kltvariety}, using the product formulas proved
in the next section.

In dimension 2, Esser's hypersurface is:
\begin{align*}
V&=V_{22}\subset \P^3(11,7,3,1),\\
0&=x_0^2+x_1^3x_3+x_2^5x_1+x_3^{19}x_2.
\end{align*}
Here $V/\mu_{13}$ is a klt Calabi--Yau surface of mld $1/13$,
which is the smallest possible \cite[Proposition 6.1]{ETWCalabi}.
In dimension 3, we have:
\begin{align*}
V&=V_{191}\subset \P^4(95,61,26,8,1),\\
0&=x_0^2x_4+x_1^3x_3+x_2^5x_1+x_3^{12}x_0+x_4^{165}x_2,
\end{align*}
with an action of the cyclic group of order 311. In dimension 4, we have:
\begin{align*}
V&=V_{925594}\subset \P^5(462797,308531,132129,21445,691,1),\\
0&=x_0^2+x_1^3x_5+x_2^7x_4+x_3^{37}x_2+x_4^{893}x_1+x_5^{904149}x_3,
\end{align*}
with an action of the cyclic group of order 677785.

\section{Product formulas for the small-mld example}

Here we prove some product formulas which imply
the desired properties of the klt Calabi--Yau variety with small mld
\cite[equation (6)]{Essermld}.
The formulas are also useful
for the klt Calabi--Yau variety with large index,
because the equations for the two varieties are similar.

Fix a positive integer $r$. Let $b_0,\ldots,b_{2r+1}$ be the numbers
defined in section \ref{esserintro}. (These are all but the last
exponent of Esser's hypersurface $V$ of dimension $2r+1$.) In the notation
of that section, define the following related numbers for $0\leq k\leq r+1$:
\begin{align*}
g_k& \coloneqq B_{k,2r+1-k,\ldots,r,r+1}\\
t_k& \coloneqq B_{r+1,r,\ldots,2r+1-k,k}\\
w_k& \coloneqq (s_{k}-1)[B_{r+1,r,\ldots,2r+1-k,k}-B_{r,\ldots,2r+1-k,k}
+\cdots-B_{k}+1]-1.
\end{align*}
\begin{proposition}
\label{product}
For $0\leq k\leq r+1$, we have
$$(s_k-1)g_kt_k-1=b_{k}\cdots b_{2r+1-k}w_k.$$
\end{proposition}

For $k=0$, we'll see in section 
\ref{kltvarietysection} that this proposition 
directly implies Esser's conjecture for the small-mld example
of dimension $n=2r+1$; in the notation
of section \ref{esserintro}, the $k = 0$ 
statement reads $m_{2r+1}u_{2r+1}-1
=b_0\cdots b_{2r+1}v_{2r+1}$. Likewise, for $k=1$, the proposition
reads $2m_{2r}u_{2r}-1=b_1\cdots b_{2r}v_{2r}$ in the notation
of section \ref{esserintro}; this will
imply Esser's conjecture for dimension $2r$. 
Generalizing these product formulas
to Proposition \ref{product} makes an inductive proof possible.

\begin{proof}
(Proposition \ref{product})
We prove this by descending induction on $0\leq k\leq r+1$.
For $k=r+1$, both sides of the equation are equal
to $(s_{r+1}-1)-1$. Next, assume that $0\leq k\leq r$
and the equation holds for $k+1$,
meaning that
\begin{equation}
\label{kplus1}
(s_{k+1}-1)g_{k+1}t_{k+1}-1
=b_{k+1}\cdots b_{2r-k}w_{k+1}.
\end{equation}
We will prove it for $k$.

We prove the following lemma at the same time as Proposition \ref{product}.

\begin{lemma}
\label{qk}
For $0\leq k\leq r$,
$$g_k-s_kg_{k+1}=(s_k-1)b_{k+1}\cdots b_{2r-k}w_{k+1}.$$
\end{lemma}

\begin{proof}
Given that Proposition \ref{product} holds for $k+1$, we prove this lemma
for $k$. By definition of $g_k$, we have
$$g_k=b_kb_{2r+1-k}g_{k+1}-(b_k-1).$$
So, noting that $b_k=s_k$ (the Sylvester number), we have
\begin{align*}
g_k-s_kg_{k+1}&=s_k(b_{2r+1-k}-1)g_{k+1}-(s_k-1).
\end{align*}
By definition of $b_{2r+1-k}$, we have $b_{2r+1-k}-1=(s_k-1)^2t_{k+1}$.
So 
\begin{align*}
g_k-s_kg_{k+1}&=s_k(s_k-1)^2g_{k+1}t_{k+1}-(s_k-1)\\
&=(s_k-1)[(s_{k+1}-1)g_{k+1}t_{k+1}-1]\\
&=(s_k-1)b_{k+1}\cdots b_{2r-k}w_{k+1},\\
\end{align*}
using that Proposition \ref{product} holds for $k+1$ (equation \ref{kplus1}).
That proves the lemma for $k$.
\end{proof}

We continue the proof of Proposition \ref{product} for $k$,
using that it holds for $k+1$. By definition of $t_k$, we have
$$t_k-t_{k+1}=(s_k-1)b_{k+1}\cdots b_{2r+1-k}.$$
Using Lemma \ref{qk} for the given number $k$, it follows that
\begin{multline*}
(s_k-1)g_kt_k-1=(s_k-1)[s_kg_{k+1}+(s_k-1)b_{k+1}\cdots b_{2r-k}w_{k+1}]\\
\cdot [t_{k+1}+(s_k-1)b_{k+1}\cdots b_{2r+1-k}]-1\\
=(s_{k+1}-1)g_{k+1}t_{k+1}-1+(s_k-1)^2b_{k}\cdots b_{2r+1-k}
g_{k+1}\\
+(s_k-1)^2b_{k+1}\cdots b_{2r-k}t_{k+1}w_{k+1}
+(s_k-1)^3(b_{k+1}\cdots b_{2r-k})^2b_{2r+1-k}w_{k+1}.
\end{multline*}
Since Proposition \ref{product} holds for $k+1$ (equation
\ref{kplus1}), it follows that
\begin{multline*}
(s_k-1)g_kt_k-1=b_{k+1}\cdots b_{2r-k}[(s_k-1)^2b_{k}b_{2r+1-k}g_{k+1}
+w_{k+1}+(s_k-1)^2t_{k+1}w_{k+1}\\
+(s_k-1)^3b_{k+1}\cdots b_{2r+1-k}w_{k+1}].
\end{multline*}
By definition, $b_{2r+1-k}=1+(s_k-1)^2t_{k+1}$. So we can combine
the second and third terms in the bracket above
into a multiple of $b_{2r+1-k}$:
\begin{multline*}
(s_k-1)g_kt_k-1=b_{k+1}\cdots b_{2r+1-k}[(s_k-1)^2b_{k}g_{k+1}+w_{k+1}\\
+(s_k-1)^3b_{k+1}\cdots b_{2r-k}w_{k+1}].
\end{multline*}
Therefore, to prove Proposition \ref{product} for $k$
(completing the induction), it suffices to show that:
\begin{equation}
\label{last}
b_kw_k=(s_k-1)^2b_kg_{k+1}+w_{k+1}+(s_k-1)^3b_{k+1}\cdots b_{2r-k}w_{k+1}.
\end{equation}

By definition,
$$w_k=(s_{k}-1)[B_{r+1,r,\ldots,2r+1-k,k}-B_{r,\ldots,2r+1-k,k}
+\cdots-B_{k}+1]-1$$
and
$$w_{k+1}=s_k(s_{k}-1)[B_{r+1,r,\ldots,2r-k,k+1}-B_{r,\ldots,2r-k,k+1}
+\cdots-B_{k+1}+1]-1.$$
Therefore, subtracting term by term and using that $b_k=s_k$, we have
\begin{multline*}
b_kw_k-w_{k+1}=s_k(s_k-1)[(b_k-1)b_{2r+1-k}b_{k+1}b_{2r-k}\cdots b_rb_{r+1}\\
-(b_k-1)b_{2r+1-k}b_{k+1}b_{2r-k}\cdots b_r+\cdots +(b_k-1)b_{2r+1-k}
-(b_k-1)+1]-(s_k-1)
\end{multline*}
\begin{align*}
&=(s_k-1)^2[b_kb_{2r+1-k}b_{k+1}b_{2r-k}\cdots b_rb_{r+1}
-b_kb_{2r+1-k}\cdots b_r+\cdots +b_kb_{2r+1-k}-b_k+1]\\
&=(s_k-1)^2B_{k,2r+1-k,\ldots,r,r+1}\\
&=(s_k-1)^2g_k.
\end{align*}
So the left side of \eqref{last} is
$$b_kw_k=(s_k-1)^2g_k+w_{k+1}.$$
By Lemma \ref{qk} for the given number $k$, we can expand $g_k$ here,
giving that
$$b_kw_k=(s_k-1)^2b_kg_{k+1}+w_{k+1}+(s_k-1)^3b_{k+1}\cdots
b_{2r-k}w_{k+1}.$$
That proves equation \ref{last}. Thus Proposition \ref{product}
holds for $k$ given that it holds for $k+1$. The proposition
is proved.
\end{proof}

\section{Proof of the properties of Esser's example
in all dimensions}
\label{kltvarietysection}

Prior to this paper, the properties of Esser's klt 
Calabi--Yau variety with small mld
were known completely only
in dimensions at most 18. We now prove the desired properties
in all dimensions (Theorem \ref{kltvariety}),
using the product formulas in Proposition \ref{product}.

\begin{theorem}
\label{kltvariety}
In each dimension $n\geq 2$,
Esser's hypersurface $V$ defined in section \ref{esserintro}
has an action of the cyclic group $\mu_m$ (for the number
$m=m_n$ defined there) such that $V/\mu_m$
is a complex klt Calabi--Yau variety
with mld $1/m$. In particular, $m_n>2^{2^n}$ for $n>2$.
\end{theorem}

We also compute the weights and degree of Esser's hypersurface explicitly,
in particular proving Esser's conjecture
that the last weight $a_{n+1}$ is equal to 1.

\begin{proof}
Let $r$ be a positive integer, and let $n$ be $2r$ or $2r+1$.
We first prove the basic properties
of Esser's hypersurface $V$, namely that it is a well-formed
quasi-smooth Calabi--Yau hypersurface.
Here $V$ was defined in section
\ref{esserintro} by writing out its equation.
That determines the weights $a_0,\ldots,a_{n+1}$ and degree $D$ of $V$,
given the requirement that
$\gcd(a_0,\ldots,a_{n+1})=1$.
By the form of the equation (a loop for $n$ odd, $x_0^2$ plus a loop
for $n$ even), $V$ is quasi-smooth.

The weighted projective space containing $V$
is well-formed: if a prime number $p$
divides all but one weight $a_i$, then $p$ divides $D$ since there
is a monomial not involving $x_i$.  Then,
there is either a monomial of the form
$x_j^{a}x_i$ in the equation for $V$ or
$n$ is even and $i = 0$.  In the former case,
we get that $p$ divides $a_i$ too,
contradicting that $\gcd(a_0,\ldots,a_{n+1})=1$. 
The same holds in the latter case unless $p = 2$.
If $p = 2$ divides all weights except $a_0$
in the $n$ even case,
then $D \equiv 2 \pmod 4$.
The remaining exponents $b_1,\ldots,b_{2r},v_{2r}$
are odd, so this means the weights
$a_i$ would have to alternate between $0 \pmod 4$
and $2 \pmod 4$ moving around the loop part of
the equation of $X$.
But the loop has odd length, a contradiction.
It follows that $V$ is well-formed
by Proposition \ref{iano}, together
with a look at the equation for $V$ in dimension 2
(section \ref{esserintro}).

\begin{lemma}
\label{esser-cy}
The hypersurface $V$ is Calabi--Yau,
in the sense that $D=\sum a_j$.
\end{lemma}

\begin{proof}
Let $A$ be the $(n+2)\times (n+2)$ matrix that encodes the equation
of $V$ (with each row specifying the exponents of one monomial
in the equation). Define the {\it charges }$q_0,\ldots,q_{n+1}$
as the sums of the rows of the matrix $A^{-1}$. Then the degree $D$
of $V$ is the least common denominator of the charges $q_i$,
and the weights of $V$ are given by $a_i=Dq_i$ \cite[section 2.3]{Essermld}.
To show that $V$ is Calabi--Yau, we have to show that the sum
of the charges is 1, that is, that the sum of the entries
of $A^{-1}$ is 1. We use the formula for the inverse
of a loop matrix \cite[Lemma 2.7]{Essermld}:

\begin{lemma}
\label{inverse}
Let $A$ be the loop matrix
$$A=\begin{pmatrix} e_1 & 1 & & \\
 & e_2 & \ddots & \\
 & & \ddots & 1\\
1 & & & e_k
\end{pmatrix}$$
with $k$ odd. Then the inverse matrix $A^{-1}$ is
$$\frac{1}{e_1\cdots e_k +(-1)^{k-1}}\begin{pmatrix} e_2\cdots e_k
& -e_3\cdots e_k & \cdots & -e_k & 1\\
1 & e_3\cdots e_ke_1 & \cdots & e_ke_1 & -e_1\\
\vdots &\vdots &\vdots & \vdots & \vdots\\
e_2\cdots e_{k-2} & \cdots & 1 & e_ke_1\cdots e_{k-2} & -e_1\cdots e_{k-2}\\
-e_2\cdots e_{k-1} & \cdots & -e_{k-1} & 1 & e_1\cdots e_{k-1}
\end{pmatrix}.$$
\end{lemma}

For $n=2r+1$, it follows that $A^{-1}$ is an integer matrix divided by
$b_0\cdots b_{2r+1}v_{2r+1}+1$. So we want to show
that $b_0\cdots b_{2r+1}v_{2r+1}+1$ minus the sum 
of the entries of that integer matrix is 0. Analyze this difference
by collecting all terms that are multiples of $b_{r+1}\cdots b_{2r+1}v_{2r+1}$,
then the remaining terms
that are multiples of $b_{r+2}\cdots b_{2r+1}v_{2r+1}$,
and so on. The result is:
\begin{multline*}
b_{r+1}\cdots b_{2r+1}v_{2r+1}\bigg[b_0\cdots b_r
\bigg( 1-\frac{1}{b_0}-\cdots-\frac{1}{b_r}\bigg)\bigg]\\
-\sum_{i=1}^{r+1}b_{r+1+i}\cdots b_{2r+1}v_{2r+1}
[b_0\cdots b_{r-i}(b_{r+1-i}-1)
B_{r+1,r,\ldots,r-1+i,r+2-i}]\\
-[B_{r+1,r,\ldots,2r+1,0}-B_{r,\ldots,2r+1,0}+\cdots-B_0+1]+1.
\end{multline*}
The first term in brackets is 1, by the properties of the Sylvester
numbers $b_0,\ldots,b_r$. By definition of $b_{r+i}$,
the expression
in brackets in the term indexed by $i$ (for $1\leq i\leq r+1$)
is $b_{r+i}-1$. Finally, the last line is $-v_{2r+1}$.
Therefore, the sum telescopes to zero. Thus
$V$ is Calabi--Yau in the sense that $D=\sum a_j$.

For $n=2r$, the proof is similar. Again, we have to show
that the sum of the entries of the matrix $A^{-1}$ is 1.
In this case, the equation of $V$ is $x_0^2$ plus a loop. By
Lemma \ref{inverse}, $A^{-1}$
is the block matrix $1/2$ followed by $1/(b_1\cdots b_{2r}v_{2r}+1)$
times an integer matrix $C$. So we need to show that the sum of the entries
of $C$ is $(b_1\cdots b_{2r}v_{2r}+1)/2$. Explicitly,
$b_1\cdots b_{2r}v_{2r}+1$ minus 2 times the sum of the entries of $C$
(which we want to be zero) is
\begin{multline*}
b_{r+1}\cdots b_{2r}v_{2r}\bigg[b_1\cdots b_r\bigg(
1-2\bigg(\frac{1}{b_1}-\cdots-\frac{1}{b_r}\bigg)\bigg)\bigg]\\
-\sum_{i=1}^{r}b_{r+1+i}\cdots b_{2r}v_{2r}[2b_1\cdots b_{r-i}(b_{r+1-i}-1)
B_{r+1,r,\ldots,r-1+i,r+2-i}]\\
-2[B_{r+1,r,\ldots,2r,1}-B_{r,\ldots,2r,1}+\cdots-B_1+1]+1.
\end{multline*}
Again, the first term in brackets is 1,
the expression in brackets in the term indexed by $i$
is $b_{r+i}-1$, and the last line is $-v_{2r}$.
So the sum telescopes to zero. Thus $V$ is Calabi--Yau
in the sense that $D=\sum a_j$,
in even as well as odd dimensions. Lemma \ref{esser-cy} is proved.
\end{proof}

By Lemma \ref{esser-cy} plus the preceding results,
$V$ is a well-formed quasi-smooth Calabi--Yau hypersurface.
More precisely, $K_V=O_V(D-\sum a_j)=O_V$,
and so $K_V$ is linearly equivalent
to zero. As a result,
$K_V$ is Cartier and $V$ is klt,
so $V$ is canonical.

The inverse $A^{-1}$ of a loop matrix is written explicitly
in Lemma \ref{inverse}. We read off that the charge $q_{2r+2}$
is the alternating sum
\begin{align*}
q_{2r+2}&=\frac{B_{0,2r+1,\ldots,r,r+1}}{b_0\cdots b_{2r+1}v_{2r+1}+1}\\
&= \frac{m_{2r+1}}{b_0\cdots b_{2r+1}v_{2r+1}+1}\\
&= \frac{1}{u_{2r+1}},
\end{align*}
by Proposition \ref{product}. The degree $D$ of $V$ is the least
common multiple of the denominators of the charges $q_i$;
but in this case of a loop matrix, all the charges have the same
denominator \cite[Lemma 7.1]{ETWCalabi}. Therefore, the degree $D$ of $V$
is equal to $u_{2r+1}$. So the weights of $V$ are
$a_i=u_{2r+1}q_i$. In particular, the weight $a_{2r+2}$ is equal to 1,
as we want.

By the loop equation of $V$, the absolute value
of the determinant of $A$ is given by \cite[section 3]{ABS}:
$$|\det(A)|=b_0\cdots b_{2r+1}v_{2r+1}+1.$$
The group $\Aut_T(V)$ of {\it toric automorphisms }of $V$ is the group of
automorphisms of $V$ that are given by diagonal matrices in the given
coordinates. This is related to the degree $D$ of $V$ by
\cite[section 3]{ABS}:
$$|\det(A)|=D|\Aut_T(V)|.$$
By the computation of the degree $D$ above,
it follows that $|\Aut_T(V)|=m_{2r+1}$. Moreover,
since the equation of $V$ is a loop, the group
$\Aut_T(V)$ is cyclic \cite[Proposition 2]{ABS}. So the group of toric
automorphisms of $V$ is the cyclic group $\mu_m$
with $m=m_{2r+1}$. (We will describe the $\mu_m$-action
explicitly in the proof of Lemma \ref{crepant}.)

Also, because the equation of $V$
is a loop, the action of $\mu_m$
on $V$ is free in codimension 1
\cite[Proposition 7.2]{ETWCalabi}.
As a result, $V/\mu_m$ is a klt Calabi--Yau variety (not a pair).

Esser showed that the mld of $V/\mu_m$ is the smallest charge
of the BHK mirror of $V$ \cite[Theorem 3.1]{Essermld}.
(The mirror is defined to be the hypersurface
whose equation is associated to the transpose
of the matrix $A$.) One mirror charge is
\begin{align*}
q_{2r+2}^{\mathsf{T}} & =\frac{B_{r+1,r,r+2,r-1,\ldots,2r+1,0}}{b_0\cdots b_{2r+1}v_{2r+1}+1}\\
&= \frac{u_{2r+1}}{b_0\cdots b_{2r+1}v_{2r+1}+1}\\
&= \frac{1}{m_{2r+1}}.
\end{align*}
Since $A^{\T}$ is a loop matrix, all the mirror charges have the same
denominator, and so $1/m_{2r+1}$ is the smallest mirror charge.
Thus $\mld(V/\mu_m)=1/m$, as we want. (Another way to compute the mld
is given in Lemma \ref{crepant}: the variety $V/\mu_m$ is crepant-birational
to the pair $(X,(1-\frac{1}{m})S)$ discussed later.)

It remains to give the analogous argument in even dimensions.
Let $n=2r$. 
Because the equation of $V$ is of the form $x_0^2$ plus a loop,
we can write out the inverse matrix $A^{-1}$ using the formula
for the inverse of a loop matrix, Lemma \ref{inverse}.
We read off that the charge $q_{2r+1}$
is the alternating sum
\begin{align*}
q_{2r+1}&=\frac{B_{1,2r,\ldots,r,r+1}}{b_1\cdots b_{2r}v_{2r}+1}\\
&=\frac{m_{2r}}{b_1\cdots b_{2r}v_{2r}+1}\\
&=\frac{1}{2u_{2r}},
\end{align*}
by Proposition \ref{product}. 
The degree $D$ of $V$ is the least
common multiple of the denominators of the charges $q_i$;
so the degree $D$ is a multiple of $2u_{2r}$, say $D=2u_{2r}\lambda$
for a positive integer $\lambda$. The last weight $a_{2r+1}$ is then
$Dq_{2r+1}=\lambda$. Using the loop of monomials (starting with
$x_{2r+1}^{v_{2r+1}}x_{r+1}$) in the equation of $V$, it follows
that all weights $a_i$ with $i\neq 0$ are also multiples of $\lambda$.
By the monomial $x_0^2$ in the equation of $V$,
we have $a_0=D/2=u_{2r}\lambda$, which is also
a multiple of $\lambda$. Since $\gcd(a_0,\ldots,a_{2r+1})=1$, it follows
that $\lambda=1$. Thus the degree $D$ of $V$
is equal to $2u_{2r}$, the weights of $V$ are
$a_i=2u_{2r}q_i$, and the weight $a_{2r+1}$ is equal to 1,
as we want.

Since the equation of $V$ for $n=2r$ is $x_0^2$ plus
a loop, the absolute value
of the determinant of $A$ is given by \cite[section 3]{ABS}:
$$|\det(A)|=2(b_1\cdots b_{2r}v_{2r}+1).$$
The group of toric automorphisms is related to the degree $D$ of $V$ by
\cite[section 3]{ABS}:
$$|\det(A)|=D|\Aut_T(V)|.$$
By the computation of the degree $D$ above,
it follows that $|\Aut_T(V)|=2m_{2r}$. Moreover,
since the equation of $V$ is $x_0^2$ plus a loop, the group
$\Aut_T(V)$ is $\mu_2$ times a cyclic group \cite[Proposition 3.2]{ABS}.
So the group of toric
automorphisms of $V$ is the product group $\mu_2\times\mu_m$
with $m=m_{2r}$. (We will describe the $\mu_m$-action
explicitly in Lemma \ref{crepant}.
The $\mu_2$-action changes the sign of $x_0$.)

Also, because the equation of $V$
is $x_0^2$ plus a loop, the action of $\mu_m$
on $V\cap \{x_0\neq 0\}$ is free in codimension 1
\cite[Proposition 7.2]{ETWCalabi}. Also, a toric automorphism
that fixes the divisor $V\cap \{x_0=0\}$ can be written
as $[x_0,\ldots,x_{n+1}]\mapsto [ex_0,x_1,\ldots,x_{n+1}]$
for some $e\in \C^*$. We have $e^2=1$, because the automorphism
preserves the equation $x_0^2+x_1^3x_{2r+1}+\cdots=0$ of $V$.
Since $m=m_{2r}$
is odd, it follows that no nontrivial element of $\mu_m$ fixes
this divisor. So $\mu_m$ acts freely in codimension 1 on $V$,
and hence $V/\mu_m$ is a klt Calabi--Yau variety (not a pair).
Given that, Esser showed that $\mld(V/\mu_m)=1/m$
\cite[section 4]{Essermld}. (Another way to compute the mld
is given in Lemma \ref{crepant}:
the variety $V/\mu_m$ is crepant-birational
to the pair $(X,(1-\frac{1}{m})S)$ discussed later.)
\end{proof}

\section{On the first gap of global log canonical thresholds}
\label{pairmldsection}

\begin{theorem}
\label{pairmld}
For every $n\geq 2$, there is a complex klt Calabi--Yau pair
$(X,(1-\frac{1}{m})S)$
of dimension $n$ with $S$ an irreducible divisor,
and with the number $m=m_n$
defined in section \ref{esserintro}. In particular, $m_n>2^{2^n}$ for $n>2$.
\end{theorem}

In principle, this follows from Theorem \ref{kltvariety}, with $X$ some
birational model of Esser's klt Calabi--Yau variety $V/\mu_m$;
but the point of this section is to construct an explicit
variety $X$.
This example should be optimal. A bit more strongly, we conjecture:

\begin{conjecture}
\label{pairconj}
For every $n\geq 2$, if $(X,(1-b)S)$ is a complex klt Calabi--Yau pair
of dimension $n$ such that $S$ is a nonzero effective
Weil divisor, then $b\geq 1/m_n$ for the number $m_n$
in Theorem \ref{pairmld}.
\end{conjecture}

The conjecture is true in dimension 2, by Liu and Shokurov
\cite[Theorem 1.1]{LS}. They formulate several other
extremal problems which have the same bound in dimension 2
(namely, $1/13$), and conjecturally in all dimensions.
In dimension 1, the bound in Conjecture \ref{pairconj}
is $1/3$, by the example of $(\P^1,\frac{2}{3}S)$
with $S$ equal to 3 points.
We know that there is some positive lower bound for this problem
in each dimension,
by Hacon-M\textsuperscript{c}Kernan-Xu \cite[Theorem 1.5]{HMXACC}.

In dimension 2, our example is:
\begin{align*}
X&=X_{30}\subset \P^3(15,10,4,13),\\
0&=x_0^2+x_1^3+x_2^5x_1+x_3^2x_2.
\end{align*}
Here $S=X\cap \{x_3=0\}$, and $(X,\frac{12}{13}S)$ is the desired Calabi--Yau
pair. In dimension 3, we have:
\begin{align*}
X&=X_{360}\subset \P^4(180,115,49,15,311),\\
0&=x_0^2+x_1^3x_3+x_2^5x_1+x_3^{12}x_0+x_4x_2.
\end{align*}
Here $S=X\cap \{x_4=0\}$, and $(X,\frac{310}{311}S)$ is the desired Calabi--Yau
pair. In dimension 4, we have:
\begin{align*}
X&=X_{1387722}\subset \P^5(693861,462574,198098,32152,1036,677785),\\
0&=x_0^2+x_1^3+x_2^7x_4+x_3^{37}x_2+x_4^{893}x_1+x_5^2x_3.
\end{align*}
Here $S=X\cap \{x_5=0\}$, and $(X,\frac{677784}{677785}S)$
is the desired Calabi--Yau pair.

\begin{proof}
(Theorem \ref{pairmld})
We want to construct a klt Calabi--Yau pair
$(X,(1-\frac{1}{m})S)$ of dimension $n\geq 2$ with $S$ an irreducible divisor
(and the number $m=m_n$ defined above).
We define $X$ as a hypersurface in a weighted projective space
$Y=\P^{n+1}(c_0,\ldots,c_{n+1})$.
Namely, if $n=2r+1$, let $X$ have the same equation as $V$
(from section \ref{esserintro}) except
for the first and last terms:
$$0=x_0^2+x_1^{b_1}x_{2r+1}+\cdots+x_r^{b_r}x_{r+2}
+x_{r+1}^{b_{r+1}}x_r
+\cdots+x_{2r+1}^{b_{2r+1}}x_0
+x_{2r+2}x_{r+1}.$$
We also have formulas for the degree $d$ and the weights $c_j$ of $X$.
(With these definitions,
it is straightforward that each monomial above
has degree $d$.)
First, $d \coloneqq b_0\cdots b_{2r+1}$. Next, the last weight $c_{2r+2}$
is $m_{2r+1}=B_{0,2r+1,1,2r,\ldots,r,r+1}$, the index
of Esser's example.
Finally,
$$c_{r+1+i} \coloneqq b_{r+1-i}\cdots b_r b_{r+1}\cdots b_{r+i}B_{0,2r+1,1,2r,\ldots,
r-i}$$
for $0\leq i\leq r$, and
$$c_{r-i} \coloneqq b_{r+1-i}\cdots b_rb_{r+1}\cdots b_{r+1+i}B_{0,2r+1,1,2r,\ldots,
r-1-i,r+2+i}$$
for $0\leq i\leq r$. (Since a $B$ with empty subscript
is 1, we have $c_0=b_1\cdots b_{2r+1}$.)

If $n=2r$, 
let $X$ have the same equation as $V$ except
for the second and last terms:
$$0=x_0^2+x_1^3+x_2^{b_2}x_{2r}+\cdots+x_r^{b_r}x_{r+2}
+x_{r+1}^{b_{r+1}}x_r
+\cdots+x_{2r}^{b_{2r}}x_1
+x_{2r+1}^{2}x_{r+1}.$$
Again, we have formulas for the degree $d$ and weights $c_j$ of $X$.
First, $d \coloneqq b_0\cdots b_{2r}$. Next, the last weight $c_{2r+1}$
is $m_{2r}=B_{1,2r,2,2r-1,\ldots,r,r+1}$, the index 
of Esser's example.
Also, $c_0 \coloneqq b_1\cdots b_{2r}$.
Finally,
$$c_{r+1+i} \coloneqq 2b_{r+1-i}\cdots b_r b_{r+1}\cdots b_{r+i}B_{1,2r,2,2r-1,\ldots,
r-i}$$
for $0\leq i\leq r-1$, and
$$c_{r-i} \coloneqq 2b_{r+1-i}\cdots b_rb_{r+1}\cdots b_{r+1+i}B_{1,2r,2,2r-1,\ldots,
r-1-i,r+2+i}$$
for $0\leq i\leq r-1$.

Define the divisor $S$ to be $X\cap \{x_{n+1}=0\}$. From the form
of the equation (a chain as in section \ref{notation},
or $x_0^2$ plus a chain),
it is immediate
that $X$ and $S$ are quasi-smooth. (Since $S$ is a hypersurface
of dimension at least 1 in a weighted projective space,
it follows that $S$ is irreducible.)
The weight $c_{n+1}$ (the degree of $S$)
is equal to the index $m$ of Esser's example,
Next, we show that $K_X=O_X(d-\sum c_j)$ is equal to $O_X(1-m)$,
so that $(X,(1-\frac{1}{m})S)$ is a klt Calabi--Yau pair (assuming
that $X$ is well-formed, as we show below).
Since $c_{n+1}=m$, it is equivalent to show that
$d-\sum_{j=0}^n c_j=1$. Assume that $n=2r+1$. We compute
$d-\sum_{j=0}^n c_j$ by collecting all terms (in the formulas
above for $d$ and the $c_j$'s) that are multiples of $b_{r+1}\cdots b_{2r+1}$,
then the remaining terms that are multiples of $b_{r+2}\cdots b_{2r+1}$,
and so on. The result is:
\begin{multline*}
d-\sum_{j=0}^{2r+1} c_j=b_{r+1}\cdots b_{2r+1}\bigg[b_0\cdots b_r
\bigg( 1-\frac{1}{b_0}-\cdots-\frac{1}{b_r}\bigg)\bigg]\\
-\sum_{i=1}^{r+1}b_{r+1+i}\cdots b_{2r+1}[b_0\cdots b_{r-i}(b_{r+1-i}-1)
B_{r+1,r,\ldots,r-1+i,r+2-i}].
\end{multline*}
The first term in brackets is 1, by the properties of the Sylvester
numbers $b_0,\ldots,b_r$. By definition of $b_{r+i}$,
the expression
in brackets in the term indexed by $i$ (for $1\leq i\leq r+1$)
is $b_{r+i}-1$. Therefore, the sum telescopes, showing that
$d-\sum_{j=0}^{2r+1}c_j=1$, as we want. For $n=2r$, the proof is similar:
\begin{multline*}
d-\sum_{j=0}^{2r} c_j=b_{r+1}\cdots b_{2r}\bigg[ b_0\cdots b_r
\bigg( 1-\frac{1}{b_0}-\cdots-\frac{1}{b_r}\bigg)\bigg]\\
-\sum_{i=1}^{r}b_{r+1+i}\cdots b_{2r}[2b_1\cdots b_{r-i}(b_{r+1-i}-1)
B_{r+1,r,\ldots,r-1+i,r+2-i}].
\end{multline*}
Again, the expression in brackets in the term indexed by $i$
is $b_{r+i}-1$. So the sum telescopes, showing that
$d-\sum_{j=0}^{2r}c_j=1$, as we want.

The only property that remains to be checked
is that $X$ is well-formed. (We need this in order to justify the formula
above for $K_X$, with $X$ viewed as a variety rather than a pair.)
The main point is to show that the weighted
projective space $Y=\P^{n+1}(c_0,\ldots,c_{n+1})$ is well-formed.
Indeed, from there it follows that $X$ is well-formed when $n\geq 3$,
by Proposition \ref{iano}.
This general result
does not apply in dimension 2,
but in that case we can check by hand that our example
is well-formed. Namely, the example is
$X\subset Y=\P^3(15,10,4,13)$, defined by
$0=x_0^2+x_1^3+x_2^5x_1+x_3^2x_2$. This is well-formed, because
$X$ does not
contain the 1-dimensional stratum $\{ x_2=x_3=0\}$ of $Y$ with stabilizer
$\mu_5$ or the 1-dimensional stratum $\{ x_0=x_3=0\}$ with stabilizer
$\mu_2$.

So, for each $n\geq 2$, we need to show that $Y$ is well-formed, meaning
that $\gcd(c_0,\ldots,\widehat{c_a},\ldots,c_{n+1})=1$
for each $0\leq a\leq n+1$. Let $p$ be a prime number dividing
$c_j$ for each $j\neq a$; we will derive a contradiction.
From looking at the equation of $X$,
we see that it includes a monomial
not involving $x_a$, and so
$p$ divides the degree $d$ of $X$. For $n=2r+1$, the variable $x_a$
occurs with exponent 1 in some monomial in the equation of $X$,
and so $p$ divides $c_a$ as well. For $n=2r$, we get the same conclusion
unless $a=0$ or $a=2r+1$. In those cases, the variable $x_a$ occurs
with exponent 2 in some monomial in the equation of $X$,
and so $p$ divides $2c_a$. It follows that $p$ divides $c_a$ unless
$p=2$. To analyze the case where $n=2r$ and $p=2$, use that
$b_1,\ldots,b_{2r}$ are all odd by their definition. So
$c_0=b_1\cdots b_{2r}$ is odd and
$c_{2r+1}=m_{2r}=B_{1,2r,\ldots,r,r+1}\equiv 1-1+1-1+\cdots+1
\equiv 1\pmod{2}$, contradicting that $2$ divides all but one
of the weights $c_j$. Thus, if $p$ divides all but one of the weights
$c_j$, then it divides all the weights.

Return to allowing any dimension $n\geq 2$.
We now have a prime number $p$ that divides the weight $c_j$
for each $0\leq j\leq n+1$, and we want to get a contradiction.
For $n=2r+1$, we have $c_0=b_1\cdots b_{2r+1}$ and
$c_{2r+2}=m_{2r+1}$. By Proposition \ref{product},
$m_{2r+1}u_{2r+1}-1=b_0\cdots b_{2r+1}v_{2r+1}$, and so
$c_0$ and $c_{2r+2}$ are relatively prime.
Next, for $n=2r$, we have $c_0=b_1\cdots b_{2r}$
and $c_{2r+1}=m_{2r}$. By Proposition \ref{product},
$2m_{2r}u_{2r}-1=b_1\cdots b_{2r}v_{2r}$, and so
$c_0$ and $c_{2r+1}$ are relatively prime. This completes
the proof that $X$ is well-formed. Theorem \ref{pairmld}
is proved.
\end{proof}

We remark that $X$ is rational at least in odd dimensions,
$n=2r+1$. Indeed,
the variable $x_{2r+2}$ occurs only in one monomial, and it has exponent 1.
So the projection from $X$
to $\P^{2r+1}(c_0,\ldots,c_{2r+1})$ is a birational map.

\begin{lemma}
\label{crepant}
In each dimension $n\geq 2$, the klt Calabi--Yau pair $(X,(1-\frac{1}{m})S)$
of Theorem \ref{pairmld} is crepant-birational to Esser's
klt Calabi--Yau variety $V/\mu_m$ with mld $1/m$.
\end{lemma}

The proof also gives an explicit formula for the action
of $\mu_m$ on $V$, in terms of the weights $c_j$ of $X$.

\begin{proof}
Let $n\geq 2$. Let $W$ be the index-1 cover of the klt Calabi--Yau
pair $(X,(1-\frac{1}{m})S)$ of dimension $n$
from Theorem \ref{pairmld}. It follows that $W$ has canonical
singularities and $K_W\sim 0$. Explicitly, for $n=2r+1$,
$W$ is the Calabi--Yau hypersurface
\begin{align*}
W&=W_{d}\subset \P^{2r+2}(c_0,\ldots,c_{2r+1},1),\\
0&=x_0^2+x_1^{b_1}x_{2r+1}+\cdots+x_r^{b_r}x_{r+2}
+x_{r+1}^{b_{r+1}}x_r
+\cdots+x_{2r+1}^{b_{2r+1}}x_0
+x_{2r+2}^mx_{r+1}.
\end{align*}
The degree $d$ of $W$ is the same as for $X$, and the exponents $b_j$
and weights $c_j$ are the same except for $j=2r+2$.
The last weight is changed from $c_{2r+2}=m_{2r+1}=m$
to 1. Let the cyclic group $\mu_m$ act on $W$ with weights
$(0,\ldots,0,-1)$; then the quotient $W/\mu_m$ is the pair
$(X,(1-\frac{1}{m})S)$. (The quotient map $W\to X$
is given by $[x_0,\ldots,x_{2r+2}]\mapsto [x_0,\ldots,x_{2r+1},
x_{2r+2}^m]$.)

In even dimensions $n=2r$, the index-1 cover $W$
of $(X,(1-\frac{1}{m})S)$ has a similar
description. Here $W$ is the Calabi--Yau hypersurface
\begin{align*}
W&=W_{d}\subset \P^{2r+1}(c_0,\ldots,c_{2r},1),\\
0&=x_0^2+x_1^3+x_2^{b_2}x_{2r}+\cdots+x_r^{b_r}x_{r+2}
+x_{r+1}^{b_{r+1}}x_r
+\cdots+x_{2r}^{b_{2r}}x_1
+x_{2r+1}^{2m}x_{r+1}.
\end{align*}
Again, the last weight has been changed from $c_{2r+1}=m_{2r}=m$
to 1, and the cyclic group
$\mu_m$ acts on $W$ with weights $(0,\ldots,0,-1)$.

The equations of $W$ and Esser's hypersurface $V$ are the same
in the affine charts $x_{n+1}=1$, and the variable $x_{n+1}$ has weight 1
in both cases. Therefore, we get a birational map
$\varphi\colon W\dashrightarrow V$ by sending
$$[x_0,\ldots,x_n,1]\mapsto [x_0,\ldots,x_n,1].$$
Since $W$ and $V$ are Calabi--Yau varieties with canonical
singularities, $\varphi$ is automatically crepant.

Moreover, we can define an action of $\mu_m$ on $V$
that makes $\varphi$ $\mu_m$-equivariant. 
Indeed, we can rewrite the action of $\mu_m$ on $W$ as
\begin{align*}
\zeta([x_0,\ldots,x_n,x_{n+1}])&=[x_0,\ldots,x_n,\zeta^{-1}x_{n+1}]\\
&=[\zeta^{c_0}x_0,\ldots,\zeta^{c_n}x_n,x_{n+1}].
\end{align*}
That suggests defining an action of $\mu_m$ on $V$ by the formula:
$$\zeta([x_0,\ldots,x_n,x_{n+1}])=[\zeta^{c_0}x_0,\ldots,\zeta^{c_n}x_n,
x_{n+1}].$$
To show that this action preserves the hypersurface $V$, it suffices
to check this in the affine chart $x_{n+1}=1$; but there it is clear,
because the equation of $V$ in this chart is the same as the equation
of $W$ in the corresponding chart.
Given that, it is clear that the map $\varphi$ is $\mu_m$-equivariant.
(It suffices to check this
in the affine chart $x_{n+1}=1$, where the two $\mu_m$-actions
are given by the same formula.)

Therefore, we have a crepant birational map
from $W/\mu_m$ (viewed as a pair, namely $(X,(1-\frac{1}{m}S))$)
to the klt Calabi--Yau variety $V/\mu_m$. In particular, this shows
more explicitly why $V/\mu_m$ has mld $1/m$: because the divisor
$S$ has log discrepancy $1/m$ with respect to $V/\mu_m$.
\end{proof}

\section{Esser-Totaro-Wang's klt Calabi--Yau variety with large index}
\label{largeindexintrosection}

A major problem on boundedness of Calabi--Yau varieties is the Index
Conjecture, which says (in particular) that the index is bounded
among all klt Calabi--Yau varieties of a given dimension.
In \cite[section 7]{ETWCalabi}, the authors and Wang constructed
a klt Calabi--Yau
variety of each dimension $n \geq 2$
which conjecturally has the largest index, roughly $2^{2^n}$
in dimension $n$.  (For example, this variety has index $19$ in dimension
$2$, $493$ in dimension $3$, and $1201495$ in dimension $4$.)
In this section,
we present this example, give a clearer proof of many of its properties,
and provide a conjectural formula for the index.
We reduce the problem of computing the index in a given dimension
to showing that two explicit numbers are relatively prime,
as explained in Proposition \ref{gcdcondition}. That holds
by a computer check in dimensions at most 30.
The expected value of the index is within a constant
factor of the conjecturally largest index in the broader setting
of klt Calabi--Yau pairs with standard coefficients
(Lemma \ref{indexconstant}).

The varieties are again quotients of quasi-smooth Calabi--Yau
hypersurfaces in weighted projective space by finite groups;
the numerology of these hypersurfaces is extremely similar
to that of the small mld examples.  Indeed, for $n = 2r+1$ odd,
the hypersurface $V' \subset \P^{n+1}(a_0',\ldots,a_{n+1}')$
has the form
$$0=x_0^{b_0}x_{2r+2}+x_1^{b_1}x_{2r+1}+\cdots+x_r^{b_r}x_{r+2}
+x_{r+1}^{b_{r+1}}x_r
+\cdots+x_{2r}^{b_{2r}}x_1 + x_{2r+1}^{b'_{2r+1}}x_0
+x_{2r+2}^{v'_{2r+1}}x_{r+1}.$$
In even dimension $n=2r$,
the equation of $V'$ has the form
$$0=x_0^{b_0}+x_1^{b_1}x_{2r+1}+\cdots+x_r^{b_r}x_{r+2}
+x_{r+1}^{b_{r+1}}x_r
+\cdots+x_{2r-1}^{b_{2r-1}}x_2 + x_{2r}^{b'_{2r}}x_1
+x_{2r+1}^{v'_{2r}}x_{r+1}.$$
These equations have the same shape as the small mld examples
and share all the same exponents except for the last two,
$b'_n$ and $v'_n$. We define these last two exponents as follows:
$$\begin{cases}
b'_{2r+1} \coloneqq \frac{1}{2}(1 + b_1 \cdots b_{2r} +
(s_1-1)B_{r+1,r,\ldots,2r,1}) & \text{if }n=2r+1,\\
b'_{2r} \coloneqq \frac{1}{3}(1 + 2(s_1-1)^2 b_2 \cdots b_{2r-1} +
2(s_2-1)B_{r+1,r,\ldots,2r-1,2}) & \text{if }n=2r.
\end{cases}$$
It is straightforward to check that these two 
expressions are integers.  We note the following comparison
between $b'_n$ and $b_n$ and define constants $E = E_n$ for future use:
\begin{equation}
\label{b'comparison}
    \begin{cases}
  E_{2r+1}  \coloneqq  b'_{2r+1} - b_{2r+1} + 1 = \frac{1}{2}(b_1 \cdots b_{2r} + 1) & \text{if }n=2r+1,\\
  E_{2r}  \coloneqq  b'_{2r} - b_{2r} + 1 = \frac{1}{3}(8 b_2 \cdots b_{2r-1} + 1) & \text{if }n=2r.
\end{cases}
\end{equation}
The last exponent is given by
$$v'_{2r+1} \coloneqq B_{1,2r,\ldots,r,r+1} + (s_1-1)[B_{r+1,r,\ldots,2r,1} 
- B_{r,\ldots,2r,1} + \cdots - B_1 + 1]-1$$
if $n = 2r + 1$, and
$$v'_{2r} \coloneqq (s_1-1)^2 B_{2,2r-1,\ldots,r,r+1} + 
(s_2-1)[B_{r+1,r,\ldots,2r-1,2} - B_{r,\ldots,2r-1,2} + \cdots - B_2 + 1] - 1$$
if $n = 2r$.  Using the notation above Proposition
\ref{product}, we may write $v'_n$ more concisely
as $v'_{2r+1} = g_1 + w_1$ for $n = 2r + 1$
and $v'_{2r} = 4g_2 + w_2$ for $n = 2r$.

The weights $a_0',\ldots,a_{n+1}'$ and 
degree $D'$ of $V'$ are uniquely determined by the 
equation of $V'$, given the requirement that 
$\gcd(a_0',\ldots,a_{n+1}') = 1$.  

We set the following additional definitions:
$$u' \coloneqq \begin{cases}u'_{2r+1} = b_1 \cdots b_{2r} + (s_1-1)B_{r+1,r,r+2,r-1,\ldots,2r,1} & \text{if }n=2r+1,\\
u'_{2r} =(s_1-1)b_2 \cdots b_{2r-1} + s_1 B_{r+1,r,r+2,r-1,\ldots,2r-1,2} &\text{if }n=2r.
\end{cases}$$
Notice that $u'_{2r+1} = 2b'_{2r+1}-1$ and $4u'_{2r} = 3 b'_{2r} - 1$.
Finally, define
\begin{equation}
\label{m'definition}
m' \coloneqq \begin{cases} m'_{2r+1} \coloneqq b_0b'_{2r+1}B_{1,2r,\ldots,r,r+1} - b_0 + 1 &\text{if }n=2r+1,\\
m'_{2r} \coloneqq b_1 b'_{2r} B_{2,2r-1,\ldots,r,r+1} - b_1 + 1 &\text{if }n=2r.
\end{cases}
\end{equation}

With the setup in place, we will now prove 
several properties of $V'$.

\begin{theorem}
\label{largeindexintroproperty}
In each dimension $n \geq 2$,
the hypersurface $V' \subset \P(a_0',\ldots,a_{n+1}')$ defined
above is well-formed, quasi-smooth, and Calabi--Yau.  The 
degree $D'$ of $V'$ is given by $D' = u'_{2r+1}$ if
$n = 2r+1$ is odd, and $D' = 2u'_{2r}$ if $n = 2r$ is even.
The last two weights $a_n'$ and $a_{n+1}'$ of $V'$ equal $1$.
There is an action on $V'$ by the cyclic
group of order $m' = m'_n$, which is free in codimension $1$.
\end{theorem}

The large index example
will then be the quotient $V'/\mu_{m'}$.
We expect this quotient to have index
$m'$ in every dimension, but this is conditional
on the statement in Proposition \ref{gcdcondition}.

For $n=2$, the hypersurface $V'$ is
$$\{x_0^2+x_1^3x_3+x_2^7x_1+x_3^9x_2=0\}\subset \P^3(5,3,1,1),$$
with an action of $\mu_{19}$. Here $V'/\mu_{19}$
is a klt Calabi--Yau surface of index $19$,
which is the largest possible \cite[Proposition 6.1]{ETWCalabi}.
For $n = 3$, the hypersurface $V'$ is 
$$\{x_0^2 x_4 + x_1^3 x_3 + x_2^5 x_1 + x_3^{19} x_0 + x_4^{32} x_2 = 0\} \subset \P^4(18,12,5,1,1),$$ 
with an action of the cyclic group of order $493$.
For $n = 4$, the hypersurface $V'$ is 
$$\{x_0^2+x_1^3x_5+x_2^7x_4+x_3^{37}x_2+x_4^{1583}x_1+x_5^{2319}x_3 = 0\} \subset \P^5(1187,791,339,55,1,1),$$
with an action of the cyclic group of order $1201495$.

\begin{proof}[Proof of Theorem
\ref{largeindexintroproperty}]

Since the equation
is a loop or $x_0^2$ plus a loop, the hypersurface $V'$
is quasi-smooth.  

The weighted projective space containing $V'$
is well-formed: if a prime number $p$
divides all but one weight $a_i'$, then $p$ divides $D'$ since there
is a monomial not involving $x_i$.  Then,
there is either a monomial of the form
$x_j^{a}x_i$ in the equation for $V'$ or
$n$ is even and $i = 0$.  In the former case,
we get that $p$ divides $a_i'$ too,
contradicting that $\gcd(a_0',\ldots,a_{n+1}')=1$. 
The same holds in the latter case unless $p = 2$.
If $p = 2$ divides all weights except $a_0'$
in the $n$ even case,
then $D' \equiv 2 \pmod 4$.
The remaining exponents $b_1,\ldots,b_{2r-1},b'_{2r},v'_{2r}$
are odd, so this means the weights
$a_i'$ would have to alternate between $0 \pmod 4$
and $2 \pmod 4$ moving around the loop part of
the equation of $V'$.
But the loop has odd length, a contradiction.
It follows that $V'$ is well-formed
by Proposition \ref{iano}, together
with a look at the equation for $V'$ in dimension 2
(above).

\begin{lemma}
\label{etw-cy}
The hypersurface $V'$ is Calabi--Yau, in the sense that $D' = \sum a'_j$.
\end{lemma}

\begin{proof}
Let $A$ be the $(n+2)\times(n+2)$ matrix encoding the
equation of $V'$.  As in Lemma \ref{esser-cy},
it will be enough to show that the sum of the entries of $A^{-1}$
is $1$.  We use the description of the inverse loop matrix from
Lemma \ref{inverse}.
For $n = 2r + 1$, $A^{-1}$ is an integer matrix divided by
$b_0 \cdots b_{2r} b'_{2r+1} v'_{2r+1} + 1$.  We will
show that $b_0 \cdots b_{2r} b'_{2r+1} v'_{2r+1}$
minus the sum of entries of this integer matrix is $0$.
To do this, collect all terms that are multiples of 
$b_{r+1} \cdots b_{2r} b'_{2r+1} v'_{2r+1}$, then
the remaining terms that are multiples of 
$b_{r+2} \cdots b_{2r} b'_{2r+1} v'_{2r+1}$, and
so on.  The result is:
\begin{multline*}
b_{r+1}\cdots b_{2r}b'_{2r+1}v'_{2r+1}\bigg[b_0\cdots b_r
\bigg( 1-\frac{1}{b_0}-\cdots-\frac{1}{b_r}\bigg)\bigg]\\
-\sum_{i=1}^{r}b_{r+1+i}\cdots b_{2r}b'_{2r+1}v'_{2r+1}
[b_0\cdots b_{r-i}(b_{r+1-i}-1)
B_{r+1,r,\ldots,r-1+i,r+2-i}]\\
- v'_{2r+1}[B_{r+1,r,\ldots,2r,1}] \\ 
-[B'_{r+1,r,\ldots,2r+1,0}-B'_{r,\ldots,2r+1,0}+\cdots-B'_0+1]+1.
\end{multline*}
The notation $B'$
in the last line indicates that we have substituted
$b'_{2r+1}$ for $b_{2r+1}$ where appropriate in the
alternating sums.
We want to show that this whole expression equals $0$.
So far, the grouping of terms is the same as in the
proof of the Calabi--Yau property for the small
mld example with $n = 2r + 1$ in section
\ref{kltvarietysection}.  

The bracketed term on the first line is $1$, while
on the second it is $b_{r+i}-1$ in the $i$th summand.
On the third line, $v'_{2r+1}$ is multiplied by
$b_{2r+1}-1$ rather than $b'_{2r+1}-1$, so the terms
do not telescope quite as before.  Instead, the sum
of all but the last line gives $v'_{2r+1}(b'_{2r+1} - b_{2r+1} + 1)$.
Using \eqref{b'comparison} and also Lemma \ref{Bsymmetry}
on the last line, we have therefore reduced to showing
\begin{equation}
\label{largeindexCYcalcodd}
    \frac{b_1 \cdots b_{2r}+1}{2}v'_{2r+1} = [B'_{0,2r+1,\ldots,r,r+1}-B'_{2r+1,\ldots,r,r+1}+\cdots-B'_{r+1}+1]-1.
\end{equation}
On the left-hand side, we can rewrite $(b_1 \cdots b_{2r}v'_{2r+1})/2$
(leaving behind the remaining $v'_{2r+1}/2$) as follows,
using Proposition \ref{product} for $k = 1$:
\begin{align*}
 \frac{1}{2}b_1 \cdots b_{2r}v'_{2r+1} & = \frac{1}{2}(b_1 \cdots b_{2r}g_1 + b_1 \cdots b_{2r}w_1) \\
    &  = \frac{1}{2}(b_1 \cdots b_{2r}g_1 + (s_1-1)g_1 t_1 - 1) = g_1 \left(b'_{2r+1}-\frac{1}{2}\right) - \frac{1}{2} .
\end{align*}
On the other hand, the right-hand side becomes
\begin{align*}
    & [B'_{0,2r+1,\ldots,r,r+1}-B'_{2r+1,\ldots,r,r+1}+\cdots-B'_{r+1}+1]-1 \\
    & = b'_{2r+1} g_1 - b_0 + 1 + 1 + [B_{1,2r,\ldots,r,r+1} - B_{2r,\ldots,r,r+1} + \cdots - B_{r+1}+1]-1 \\
   &  = b'_{2r+1} g_1 + [B_{1,2r,\ldots,r,r+1} - B_{2r,\ldots,r,r+1} + \cdots - B_{r+1}+1]-1.
\end{align*}
Here we've reverted $B'$ to $B$ in sums where the 
index $2r+1$ does not appear.  We can now cancel the
terms involving $b'_{2r+1}$ from both sides of 
\eqref{largeindexCYcalcodd} and 
apply Lemma \ref{Bsymmetry}
to the bracketed term on the last line. This gives:
$$\frac{v'_{2r+1}}{2} = \frac{1}{2}g_1 + [B_{r+1,r,\ldots,2r,1} - B_{r,\ldots,2r,1} + \cdots - B_1+1] - \frac{1}{2}.$$
After multiplying by $2$, we recover the original definition
for $v'_{2r+1}$, so the proof is complete.

Now we'll show the Calabi--Yau property when $n = 2r$.
Again, we need to demonstrate that the sum
of entries of the matrix
$A^{-1}$ is $1$, where $A^{-1}$ is now a block
diagonal matrix with $1/2$ in the
top left corner followed by a
$(n+1) \times (n+1)$ block which is
$1/(b_1 \cdots b_{2r-1}b'_{2r}v'_{2r}+1)$
times an integer matrix $C$ of inverse loop form
(as in Lemma \ref{inverse}).  Explicitly,
$b_1 \cdots b_{2r-1}b'_{2r}v'_{2r} + 1$
minus $2$ times the sum of the entries of $C$
(which we want to be zero) is
\begin{multline*}
b_{r+1}\cdots b_{2r-1}b'_{2r}v'_{2r}\bigg[b_1\cdots b_r\bigg(
1-2\bigg(\frac{1}{b_1}-\cdots-\frac{1}{b_r}\bigg)\bigg)\bigg]\\
-\sum_{i=1}^{r-1}b_{r+1+i}\cdots b_{2r-1}b'_{2r}v'_{2r}[2b_1\cdots b_{r-i}(b_{r+1-i}-1)
B_{r+1,r,\ldots,r-1+i,r+2-i}]\\
-v'_{2r}[2(b_1-1)B_{r+1,r,\ldots,2r-1,2}] \\
-2[B'_{r+1,r,\ldots,2r,1}-B'_{r,\ldots,2r,1}+\cdots-B'_1+1]+1.
\end{multline*}
So far, the grouping of terms is the same as in the
proof of the Calabi--Yau property for the small
mld example with $n = 2r$ in section
\ref{kltvarietysection}.  The notation $B'$
in the last line indicates that we have substituted
$b'_{2r}$ for $b_{2r}$ where appropriate in the
alternating sums.

The bracketed term on the first line is $1$,
while it is $b_{r+i}-1$ in the $i$th summand
on the second line.  On the third line, 
$v'_{2r}$ is multiplied by $b_{2r}-1$
rather than $b'_{2r}-1$, so the sum does not
completely telescope as before.  Instead, the
sum of all but the last line gives
$v'_{2r}(b'_{2r} - b_{2r} + 1)$.  Using
\eqref{b'comparison} and also Lemma \ref{Bsymmetry}
on the last line, we have therefore reduced to showing
\begin{equation}
\label{largeindexCYcalceven}
 \frac{8 b_2 \cdots b_{2r-1} + 1}{3}v'_{2r} = 2[B'_{1,2r,\ldots,r,r+1}-B'_{2r,\ldots,r,r+1}+\cdots-B'_{r+1}+1]-1.
\end{equation}
On the left-hand side, we can rewrite $(8b_2 \cdots b_{2r-1}v'_{2r})/3$
(leaving behind the remaining $v'_{2r}/3$) as follows,
using Proposition \ref{product} for $k = 2$:
\begin{align*}
 \frac{8}{3}b_1 \cdots b_{2r}v'_{2r} & = \frac{8}{3}(b_1 \cdots b_{2r}4g_2 + b_2 \cdots b_{2r-1}w_2) \\
    &  = \frac{8}{3}(4b_2 \cdots b_{2r-1}g_2 + (s_2-1)g_2 t_2 - 1) = 4g_2 \left(b'_{2r}-\frac{1}{3}\right) - \frac{8}{3}.
\end{align*}
On the other hand, the right-hand side becomes
\begin{align*}
    & 2[B'_{1,2r,\ldots,r,r+1}-B'_{2r,\ldots,r,r+1}+\cdots-B'_{r+1}+1]-1 \\
    & = 4 b'_{2r} g_2 - 2b_1 + 2 + 2 + 2[B_{2,2r-1,\ldots,r,r+1} - B_{2r-1,\ldots,r,r+1} + \cdots - B_{r+1}+1]-1 \\
   &  = 4b'_{2r} g_2 + 2[B_{2,2r-1,\ldots,r,r+1} - B_{2r-1,\ldots,r,r+1} + \cdots - B_{r+1}+1]-3.
\end{align*}
Here we've reverted $B'$ to $B$ in sums where the 
index $2r$ does not appear.  We can now cancel the
terms involving $b'_{2r}$ from both sides of
\eqref{largeindexCYcalceven} and 
apply Lemma \ref{Bsymmetry}
to the bracketed term on the right-hand side.  This gives:
$$\frac{v'_{2r}}{3} = \frac{4}{3}g_2 + 2[B_{r+1,r,\ldots,2r-1,2} - B_{r,\ldots,2r-1,2} + \cdots - B_2+1] - \frac{1}{3}.$$
After multiplying by $3$, we recover the original definition
for $v'_{2r}$.  We've shown that $V'$ is Calabi--Yau
in every dimension. Lemma \ref{etw-cy} is proved.
\end{proof}

By Lemma \ref{etw-cy}, $K_{V'} = 
\mathcal{O}_{V'}(D' - \sum a_j') = \mathcal{O}_{V'}$.
Since the hypersurface $V'$ is klt and has trivial canonical class,
it is canonical.

Next, we'll state and prove expressions for the degree $D'$ of $V'$
and the order $m'$ of a cyclic group action on this hypersurface.
Along the way, we'll also prove that the last two weights $a_n'$
and $a_{n+1}'$ equal $1$. First, we show 
the following identities involving $m'_n$ and $u'_n$:
\begin{equation}
\label{indexidentity}
    \begin{cases} m'_{2r+1} u'_{2r+1} - 1 = b_0\cdots b_{2r}b'_{2r+1}v'_{2r+1} & \text{if }n = 2r+1,\\
2m'_{2r} u'_{2r} - 1 = b_1 \cdots b_{2r-1}b'_{2r} v'_{2r} &\text{if }n=2r.
\end{cases}
\end{equation}
First, when $n = 2r + 1$, 
\begin{multline*}
u'_{2r+1} m'_{2r+1} - 1  
= u'_{2r+1}(b_0b'_{2r+1}B_{1,2r,\ldots,r,r+1} - 1)-1 \\
    = (b_1 \cdots b_{2r} + (s_1-1)B_{r+1,r,\ldots,2r,1})b_0b'_{2r+1}B_{1,2r,\ldots,r,r+1} - u'_{2r+1}-1 \\
    = b_0 \cdots b_{2r} b'_{2r+1} B_{1,2r,\ldots,r,r+1} + b_0b'_{2r+1}(s_1 -1)B_{r+1,r,\ldots,2r,1}B_{1,2r,\ldots,r,r+1} - u_{2r+1}' - 1.
\end{multline*}
We may now apply Proposition \ref{product} with $k = 1$
to replace the expression
$(s_1 -1)B_{r+1,r,\ldots,2r,1}B_{1,2r,\ldots,r,r+1} = (s_1-1)g_1 t_1$
with $b_1 \cdots b_{2r}w_1+1$. Hence
\begin{align*}
    u'_{2r+1} m'_{2r+1} - 1 & = b_0 \cdots b_{2r} b'_{2r+1} (g_1 + w_1) + b_0 b'_{2r+1} - u_{2r+1}' - 1 \\
& = b_0 b_1 \cdots b_{2r} b'_{2r+1} v'_{2r+1} + b_0 b'_{2r+1} - u_{2r+1}' - 1 \\
& = b_0 b_1 \cdots b_{2r} b'_{2r+1} v'_{2r+1}.
\end{align*}
Looking at the form $A^{-1}$ of the loop matrix in
Lemma \ref{inverse}, we may read off the last charge
for the hypersurface $V'$ and apply \eqref{indexidentity} to get:
\begin{align*}
   q_{2r+2} & = \frac{b_0b'_{2r+1}b_1 b_{2r} \cdots b_r b_{r+1} - b_0b'_{2r+1}b_1 b_{2r} \cdots b_r + \cdots + b_0b'_{2r+1} - b_0 + 1}{b_0 b_1 \cdots b_{2r} b'_{2r+1} v'_{2r+1}+1} \\
   & = \frac{m'_{2r+1}}{b_0 b_1 \cdots b_{2r} b'_{2r+1} v'_{2r+1}+1} \\
   & = \frac{1}{u'_{2r+1}}.
\end{align*}
The degree $D'$ of $V'$ is the least common denominator of the charges.
All these denominators are the same for a loop potential,
so $D' = u'_{2r+1}$ and the weight $a'_{2r+2} = 1$.
Since $x_0^2 x_{2r+2}$ is a monomial in the equation for $V'$,
we must have $a'_0 = (u'_{2r+1}-1)/2$.  Then, since
$x_{2r+1}^{b'_{2r+1}}x_0$ is a monomial also,
$$a'_{2r+1} = \frac{u'_{2r+1}-(u'_{2r+1}-1)/2}{b'_{2r+1}} = 1.$$
Let $\mathrm{Aut}_T(V')$ be the group of toric
automorphisms of $V'$.  The order of this group is
related to the degree of the hypersurface and the matrix
$A$ by \cite[section 3]{ABS}
$$|\det(A)| = D'|\mathrm{Aut}_T(V')|.$$
Since $\det(A) = b_0 \cdots b_{2r} b'_{2r+1} v'_{2r+1} + 1$,
the identity \eqref{indexidentity} implies that 
$|\mathrm{Aut}_T(V')| = m_{2r+1}'$.  The equation of $V'$ is 
a loop, so $\mathrm{Aut}_T(V')$ is a cyclic group of order $m' = m'_{2r+1}$.
Since the equation of $V'$ is a loop, the action of $\mu_{m'}$
is free in codimension $1$ \cite[Proposition 7.2]{ETWCalabi}.

Next, if $n = 2r$, then we prove \eqref{indexidentity} as follows:
\begin{multline*}
    2 u'_{2r+1} m'_{2r+1} - 1  = 2u'_{2r+1}(b_1b'_{2r}B_{2,2r-1,\ldots,r,r+1} - 2)-1 \\
     = (2(s_1-1) b_2 \cdots b_{2r-1} + 2s_1 B_{r+1,r,\ldots,2r-1,2})b_1b'_{2r}B_{2,2r-1,\ldots,r,r+1} - 4u'_{2r+1}-1 \\
     = (s_1-1)^2 b_1\cdots b_{2r-1} b'_{2r} B_{2,2r-1,\ldots,r,r+1} \\
     + b_1b'_{2r}(s_2 -1)B_{r+1,r,\ldots,2r-1,2}B_{2,2r-1,\ldots,r,r+1} - 4u_{2r}' - 1.
\end{multline*}
We may now apply Proposition \ref{product} with $k = 2$
to replace the expression $(s_2 -1)B_{r+1,r,\ldots,2r-1,2}B_{2,2r-1,\ldots,r,r+1}
= (s_2-1)g_2 t_2$ with $b_2 \cdots b_{2r-1} w_2 +1$.  Hence
\begin{align*}
    2u'_{2r} m'_{2r} - 1 & = b_1 \cdots b_{2r-1} b'_{2r} (4g_2 + w_2) + b_1 b'_{2r} - 4u'_{2r} - 1 \\
& = b_1 \cdots b_{2r-1} b'_{2r} v'_{2r} + b_1 b'_{2r} - 4u'_{2r} - 1 \\
& = b_1 \cdots b_{2r-1} b'_{2r} v'_{2r}.
\end{align*}
For $n = 2r$, the $(n+2) \times (n+2)$ matrix $A$ expressing 
the equation of $V'$ is block diagonal with $2$ as the 
top left entry and an $(n+1) \times (n+1)$ loop matrix
as the other block.  Therefore, we may read off the last 
charge $q_{2r+1}$ from $A^{-1}$ and apply \eqref{indexidentity} to get:
\begin{align*}
    q_{2r+1} & = \frac{b_1 b'_{2r} b_2 b_{2r-1} \cdots b_r b_{r+1} - b_1 b'_{2r} b_2 b_{2r-1} \cdots b_r + \cdots + b_1 b'_{2r} - b_1 + 1}{b_1 \cdots b_{2r-1} b'_{2r} v'_{2r} + 1} \\
    & = \frac{m'_{2r}}{b_1 \cdots b_{2r-1} b'_{2r} v'_{2r} + 1} \\
    & = \frac{1}{2u'_{2r}}
\end{align*}
Therefore, the degree $D'$ is a multiple of $2u'_{2r}$, 
say $D' = 2\lambda u'_{2r}$, so the last weight is $a'_{n+1} = \lambda$.
Following the monomials around the loop, it follows that
every weight $a'_i$ with $i \neq 0$ is a multiple of $\lambda$.
This contradicts the fact that the weighted projective space
containing $V'$ is well-formed, so in fact we have
$D' = 2u'_{2r}$ and $a'_{n+1} = 1$.  Since $x_1^3 x_{2r+1}$
is a monomial of $V'$, we must have $a'_1 = (D'-1)/3 = (2u'_{2r}-1)/3$.
Then, since $x_{2r}^{b'_{2r}}x_1$ is a monomial also,
$$a'_{2r} = \frac{2u'_{2r}-(2u'_{2r}-1)/3}{b'_{2r}} = 1.$$
The determinant of the matrix $A$ in this case is
$$|\det(A)| = 2(b_1 \cdots b_{2r-1}b'_{2r}v'_{2r} + 1)$$
and $|\det(A)| = D' |\mathrm{Aut}_T(V')|$ so
\eqref{indexidentity} yields $|\mathrm{Aut}_T(V')| = 2m'_{2r}$.
Because the equation is of the form $x_0^2$ plus a loop,
$\mathrm{Aut}_T(V') \cong \mu_2 \times \mu_{m'}$ with 
$m' = m'_{2r}$.  Since the shape of the equation is
the same as in the mld example, the same argument from 
the proof of Theorem \ref{kltvariety} shows that the 
$\mu_{m'}$-action is free in codimension $1$.

This completes the proof of all the properties listed
in Theorem \ref{largeindexintroproperty}.
\end{proof}

Since the action of $\mu_{m'}$ on $V'$ is 
free in codimension $1$, the quotient 
$V'/\mu_{m'}$ is a klt Calabi--Yau variety.
The index of the quotient is determined by the 
induced action of $\mu_{m'}$ on $H^0(V',K_{V'}) \cong \C$.
In particular, the index of the quotient is $m'$
if and only if this action is faithful. This in turn
is equivalent to a concrete condition involving the
exponents.

\begin{proposition}
    \label{gcdcondition}
The quotient $V'/\mu_{m'}$ defined above has index $m' = m'_n$
in dimension $n$ if and only if $\gcd(m'_n,E_n) = 1$.
\end{proposition}

Recall that the constants $m'$ and $A$ were defined in
\eqref{m'definition} and\eqref{b'comparison}, respectively.

\begin{proof}
Suppose that $n = 2r+1$ is odd.  The determinant of the
matrix $A$ encoding the equation of $V'$ is $u'_{2r+1} m'_{2r+1}$
by \eqref{indexidentity}. The degree of the mirror
hypersurface of $V'$ always divides 
$u'_{2r+1} m'_{2r+1}/u'_{2r+1} = m'_{2r+1}$.
By \cite[Proposition 7.3]{ETWCalabi},
the action of $\mu_{m'}$ on $H^0(V',K_{V'})$ is faithful
if and only if the mirror degree actually equals $m'$.
This degree is the 
least common denominator of the mirror charges of $V'$,
which are the sums of columns of $A^{-1}$.

We may use Lemma \ref{inverse} to write the smallest mirror
charge as
\begin{align*}
    q^{\mathsf{T}}_{2r+2} & = \frac{b_{r+1}b_r \cdots b_{2r}b_1 b'_{2r+1} b_0 - b_{r+1}b_r \cdots b_{2r}b_1 b'_{2r+1} + \cdots - b_{r+1} + 1}{u'_{2r+1}m'_{2r+1}} \\
    & = \frac{b'_{2r+1} b_1 \cdots b_{2r} + B_{r+1,r,\ldots,2r,1}}{u'_{2r+1}m'_{2r+1}} \\
    & = \frac{\frac{1}{2}((2b'_{2r+1}-1)b_1 \cdots b_{2r} + b_1 \cdots b_{2r} + 2B_{r+1,r,\ldots,2r,1})}{u'_{2r+1}m'_{2r+1}} \\
    & = \frac{\frac{1}{2}(b_1 \cdots b_{2r} + 1)u'_{2r+1}}{u'_{2r+1}m'_{2r+1}} \\
    & = \frac{E_{2r+1}}{m'_{2r+1}}.
\end{align*}
If $\gcd(E_{2r+1},m'_{2r+1}) = 1$, then this proves
that the mirror degree is $m'_{2r+1}$.  Conversely,
if some prime $p$ divides $E_{2r+1}$ and $m'_{2r+1}$,
and the mirror degree were $m'_{2r+1}$, by following 
the loop potential, we'd have that $p$ divides every
weight of the mirror, a contradiction.

The same reasoning holds when $n = 2r$.  In that case,
$\mathrm{Aut}_T(V') \cong \mu_2 \times \mu_{m'}$.
The $\mu_2$-action (which sends $x_0 \mapsto -x_0$
and leaves the other variables unchanged) is faithful
on $H^0(V',K_{V'})$, so the $\mu_{m'}$-action is 
faithful if and only if the entire group $\mathrm{Aut}_T(V')$
acts faithfully. 
The determinant of the matrix $A$ encoding the equation
of $V'$ is $4m'_{2r}u'_{2r}$.
Using \cite[Proposition 7.3]{ETWCalabi} again,
$\mathrm{Aut}_T(V')$ acts faithfully
on $H^0(V',K_{V'})$ if and only if 
$2 u_{2r} D^{\mathsf{T}} =4 m'_{2r}u'_{2r}$,
where $D^{\mathsf{T}}$ is the mirror degree.
This reduces to $D^{\mathsf{T}} = 2m'_{2r}$.
We compute the smallest mirror charge as
\begin{align*}
    q^{\mathsf{T}}_{2r+1} & = \frac{b_{r+1}b_r \cdots b_{2r-1}b_2 b'_{2r} b_1 - b_{r+1}b_r \cdots b_{2r-1}b_2 b'_{2r} + \cdots - b_{r+1} + 1}{2u'_{2r}m'_{2r}} \\
    & = \frac{(s_1-1)b'_{2r} b_2 \cdots b_{2r-1} + B_{r+1,r,\ldots,2r-1,2}}{2u'_{2r}m'_{2r}} \\
    & = \frac{\frac{1}{3}((3b'_{2r}-1)(s_1-1)b_2 \cdots b_{2r-1} + (s_1-1)b_2 \cdots b_{2r-1} + 3B_{r+1,r,\ldots,2r-1,2})}{2u'_{2r}m'_{2r}} \\
    & = \frac{\frac{1}{3}(8b_2 \cdots b_{2r-1} + 1)u'_{2r}}{2u'_{2r}m'_{2r}} \\
    & = \frac{E_{2r}}{2m'_{2r}}.
\end{align*}
The constant $E_{2r}$ is odd, so if $\gcd(E_{2r},m'_{2r}) = 1$,
then the mirror degree is $2m'_{2r}$, as required.  Conversely,
if a prime $p$ (which must be odd)
divides both $E_{2r}$ and $m'_{2r}$, and the
mirror degree were $2m'_{2r}$, then the form of the equation
of $V'$ implies that all mirror weights would be divisible by
$p$, a contradiction.
\end{proof}

\begin{conjecture}
\label{conj-index}
For each integer $n\geq 2$, the numbers
$m'_n$ and $E_n$ defined above are relatively prime.
\end{conjecture}

By Proposition \ref{gcdcondition}, Conjecture \ref{conj-index}
is equivalent to the klt Calabi--Yau variety $V'/\mu_{m'}$
of dimension $n$ having index equal to $m' = m'_n$.
By computer calculation, the conjecture
holds in dimensions at most $30$.

\section{Asymptotics of the mld and index}
\label{asymptotics_section}

We show in this section that our klt Calabi--Yau varieties
of small mld or large index are within a constant factor
of the conjecturally optimal examples in the greater generality
of klt pairs with standard coefficients.

First, building on examples by Koll\'ar,
Jihao Liu constructed a klt Calabi--Yau pair of dimension $n$
with standard coefficients whose mld is $1/(s_{n+1}-1)$
\cite[Remark 2.6]{Liu}. That is conjectured to be the smallest
possible mld in this setting. Namely, Liu's pair is
\begin{equation}
(X,D) = \bigg( \P^n,\frac{1}{2}H_0+\frac{2}{3}H_1
+\frac{6}{7}H_2+\cdots+\frac{s_n-1}{s_n}H_n
+\frac{s_{n+1}-2}{s_{n+1}-1}H_{n+1}\bigg),
\end{equation}
where $H_0,\ldots,H_{n+1}$ are $n+2$ general hyperplanes in $\P^n$.

In the narrower setting
of klt Calabi--Yau varieties, it turns out that Esser's example
has mld less than 6 times $1/(s_{n+1}-1)$ in odd dimensions,
and less than 23 times this number in even dimensions.
That is extremely
close for such small numbers, and it supports the conjecture
that Esser's example has the smallest mld among klt Calabi--Yau varieties.

There is a parallel story for the problem of large index.
Wang and the authors constructed a klt Calabi--Yau pair
of dimension $n$ with standard
coefficients whose index is $(s_n-1)(2s_n-3)$
\cite[Theorem 3.3]{ETWCalabi}. That is conjectured to be the largest
possible index in this setting. We now show that the 
conjectural value for the index of the klt Calabi--Yau
variety of section \ref{largeindexintrosection} is within a constant
factor of that number (Lemma \ref{indexconstant}).

More precisely, define a constant
\begin{align*}
\alpha& \coloneqq 2\prod_{j=1}^{\infty}
\bigg[\frac{s_{j+1}}{(s_j-1)^2}\bigg]^{2^{j-1}}\\
&\doteq 5.522868.
\end{align*}
The convergence of this product is easy from the doubly exponential growth
of the Sylvester numbers $s_j$ and the fact that $s_{j+1}=(s_j-1)^2+s_j$.

\begin{lemma}
\label{constant}
For each integer $n\geq 2$, let $1/m_n$ be the mld of Esser's klt
Calabi--Yau variety of dimension $n$
(computed in Theorem \ref{kltvariety}). Then
$$\frac{1}{m_n}\leq \alpha \frac{1}{s_{n+1}-1}$$
if $n$ is even and
$$\frac{1}{m_n}\leq\frac{3\alpha^2}{4}\frac{1}{s_{n+1}-1}$$
if $n$ is odd.
\end{lemma}

The ratio $(s_{n+1}-1)/m_n$ actually converges to $\alpha$ for $n$ even
and to $3\alpha^2/4\doteq 22.876556$ for $n$ odd, as $n$ goes
to infinity; but we will not need that.

\begin{proof}
(Lemma \ref{constant})
Let $r$ be a positive integer, and let $n$ be $2r$ or $2r+1$.
Let $b_0,\ldots,b_{n+2}$ be the exponents of Esser's example,
listed in section \ref{esserintro}. Then $b_a$ is equal
to the Sylvester number $s_a$ for $a\leq r$. Therefore,
\begin{align*}
b_{r+1}-1&=(b_r-1)^2\\
&=s_{r+1}\frac{(s_r-1)^2}{s_{r+1}}.
\end{align*}
By induction on $1\leq a\leq r$, it follows that
$$b_{r+a}\geq s_{r+a}\frac{(s_{r+1-a}-1)^2}{s_{r+2-a}}
\bigg[\frac{(s_{r+2-a}-1)^2}{s_{r+3-a}}\bigg]^{2^{0}}
\cdots \bigg[\frac{(s_r-1)^2}{s_{r+1}}\bigg]^{2^{a-2}}.$$
As a result, we have
\begin{align*}
m_{2r}&=B_{1,2r,\ldots,r,r+1}\\
&=b_1b_{2r}\cdots b_rb_{r+1}-b_1b_{2r}\cdots b_r+\cdots\\
&\geq b_1\cdots b_{r}(b_{r+1}-1)b_{r+2}\cdots b_{2r}\\
&\geq s_1\cdots s_{2r}\bigg[\frac{(s_{1}-1)^2}{s_{2}}\bigg]^{2^{0}}
\cdots \bigg[\frac{(s_r-1)^2}{s_{r+1}}\bigg]^{2^{r-1}}\\
&= (s_{2r+1}-1)\frac{1}{2}\bigg[\frac{(s_{1}-1)^2}{s_{2}}\bigg]^{2^{0}}
\cdots \bigg[\frac{(s_r-1)^2}{s_{r+1}}\bigg]^{2^{r-1}}\\
&\geq \frac{1}{\alpha}(s_{2r+1}-1).
\end{align*}
The proof for $n=2r+1$ is similar. Here $m_{2r+1}=B_{0,2r+1,\ldots,
r,r+1}\geq b_0\cdots b_r(b_{r+1}-1)b_{r+2}\cdots b_{2r+1}$.
The lower bound for $b_{r+a}$ above holds (by induction)
for all $1\leq a\leq r+1$. We deduce that
\begin{align*}
m_{2r+1}&\geq s_0\cdots s_{2r+1}\frac{(s_{0}-1)^2}{s_{1}}
\bigg[\frac{(s_{1}-1)^2}{s_{2}}\bigg]^{2^{1}}
\cdots \bigg[\frac{(s_r-1)^2}{s_{r+1}}\bigg]^{2^{r}}\\
&=(s_{2r+2}-1)\frac{1}{3}
\bigg[\frac{(s_{1}-1)^2}{s_{2}}\bigg]^{2^{1}}
\cdots \bigg[\frac{(s_r-1)^2}{s_{r+1}}\bigg]^{2^{r}}\\
&\geq \frac{4}{3\alpha^2}(s_{2r+2}-1).
\end{align*}
\end{proof}

\begin{lemma}
\label{indexconstant}
For each integer $n\geq 2$, let $m'_n$ be the 
conjectural index of Esser-Totaro-Wang's klt
Calabi--Yau variety of dimension $n$ from section
\ref{largeindexintrosection}. Then
$$m'_n\geq \frac{(s_n -1)(2s_{n}-3)}{9\alpha/8}$$
if $n$ is even and
$$m'_n\geq \frac{(s_n -1)(2s_{n}-3)}{6\alpha^2/7}$$
if $n$ is odd.
\end{lemma}

Here $9\alpha/8\doteq 6.213227$ and $6\alpha^2/7\doteq 26.144635$.
Thus the expected index of the klt Calabi--Yau variety in 
section \ref{largeindexintrosection}
is within a constant factor of the conjecturally largest index
among all klt Calabi--Yau pairs with standard coefficients.

\begin{proof}
The statement is easy for $n=2$, and so we can assume that $n>2$.
The index $m'_n$ is defined in terms of the numbers
$b_0,\ldots,b_{n-1}$ (the same in the small-mld example)
together with $b'_n$ (section \ref{largeindexintrosection}).
For $n=2r$ with $r>1$, we have
\begin{align*}
b'_{2r}&=\frac{1}{3}[1+8b_2\cdots b_{2r-1}+12B_{r+1,r,\cdots,2r-1,2}]\\
&\geq \frac{1}{3}[8b_2\cdots b_{2r-1}+12(b_2-1)b_3\cdots b_{2r-1}]\\
&=\frac{128}{3}b_3\cdots b_{2r-1}\\
&\geq \frac{16}{9} s_{2r}\frac{(s_1-1)^2}{s_2}\bigg[
\frac{(s_2-1)^2}{s_3}\bigg]^{2^0}
\cdots \bigg[\frac{(s_r-1)^2}{s_{r+1}}\bigg]^{2^{r-2}},
\end{align*}
using the formula for $b_{r+1}-1$ and the lower bounds for $b_{r+a}$
from the proof of Lemma \ref{constant}.
Therefore,
\begin{align*}
m'_{2r}&=B'_{1,2r,\ldots,r,r+1}\\
&\geq b_1\cdots b_r(b_{r+1}-1)b_{r+2}\cdots b_{2r-1}b'_{2r}\\
&\geq \frac{16}{9}s_1\cdots s_{2r}\bigg[\frac{(s_1-1)^2}{s_2}\bigg]^{2^0}
\cdots \bigg[\frac{(s_r-1)^2}{s_{r+1}}\bigg]^{2^{r-1}}\\
&\geq \frac{8}{9}(s_{2r}-1)(2s_{2r}-3) \frac{1}{2}
\bigg[\frac{(s_{1}-1)^2}{s_{2}}\bigg]^{2^{0}}
\cdots \bigg[\frac{(s_r-1)^2}{s_{r+1}}\bigg]^{2^{r-1}}\\
&\geq \frac{8}{9\alpha}(s_{2r}-1)(2s_{2r}-3).
\end{align*}
We omit the similar argument for $n=2r+1$.
\end{proof}


\small \sc  Department of Mathematics, Princeton University, Fine Hall, Washington Road, Princeton, NJ 08544-1000

esserl@math.princeton.edu

\medskip

\small \sc UCLA Mathematics Department, Box 951555,
Los Angeles, CA 90095-1555

totaro@math.ucla.edu

\end{document}